\documentclass[11pt]{article}
\usepackage{fullpage}
\usepackage{amsthm}
\usepackage{amsmath}
\usepackage{amsfonts}
\usepackage{amssymb}
\usepackage{tikz-cd}
\usepackage{graphicx}
\usepackage{subcaption}
\usepackage{dsfont}
\usepackage{todonotes}

\theoremstyle{definition}
\newtheorem{thm}{Theorem}[section]
\newtheorem{lem}{Lemma}[section]
\newtheorem{prop}{Proposition}[section]
\newtheorem{cor}{Corollary}[section]
\newtheorem{ex}{Example}[section]

\newtheorem{Remark}{Remark}[section]
\DeclareMathOperator{\Ima}{Im}
\DeclareMathOperator{\rk}{rk}
\DeclareMathOperator{\id}{id}
\let\vec\mathbf
\begin{document}
\newtheorem{defn}{Definition}[section]
\title{Multiparameter Persistence Landscapes}
\author{Oliver Vipond\\  \texttt{vipond@maths.ox.ac.uk}\\
  \textsc{Mathematical Institute, University of Oxford, Oxford OX2 7DT, UK}}
\date{}
\maketitle
\setcounter{tocdepth}{1}

\begin{abstract}
    An important problem in the field of Topological Data Analysis is defining topological summaries which can be combined with traditional data analytic tools.
    In recent work Bubenik introduced the persistence landscape, a stable representation of persistence diagrams amenable to statistical analysis and machine learning tools. In this paper we generalise the persistence landscape to multiparameter persistence modules providing a stable representation of the rank invariant. 
    We show that multiparameter landscapes are stable with respect to the interleaving distance and persistence weighted Wasserstein distance, and that the collection of multiparameter landscapes faithfully represents the rank invariant. Finally we provide example calculations and statistical tests to demonstrate a range of potential applications and how one can interpret the landscapes associated to a multiparameter module.
\end{abstract}

\section{Introduction}
Topological and Geometric Data Analysis (TGDA) describes an emerging set of analytic tools which leverage the underlying shape of a data set to produce topological summaries. These techniques have been particularly successful at providing new insight for high dimensional data sets, topological data structures and biological data sets \cite{gameiro_topological_2015}\cite{kelin_multidimensional_2015}\cite{nicolau_topology_2011}. An ideal topological summary should discriminate well between different spaces, be stable to perturbations of the initial data, and amenable to statistical analysis.

Persistent homology (PH) has become a ubiquitous tool in the TGDA arsenal. PH studies the homology groups of a family of topological spaces built upon a data set. The associated topological summary for a $1$-parameter family is given in terms of a persistence diagram marking the parameter values for births and deaths of homological features. The persistent diagram may equivalently be thought of as a multiset of points in $\mathbb{R}^2$.

Qualitatively one considers long lived homological features detected as inherent to the data set and the short lived features as noise. Nevertheless in various applications it has been revealed that the short lived features provide important discriminating information in classification. For example in analysing the topology of brain arteries it was found the $28^\text{th}$ longest persisting feature provided the most useful discriminating information \cite{bendich_persistent_2016}. We would like therefore to have a statistical framework and topological summary in which one can detect statistically significant topological features of large, medium and small persistence, and in particular does not discard short lived features as noise.

The space of persistence diagrams does not enjoy desired properties for traditional statistical analysis. For example, a collection of persistence diagrams may not have a well defined mean. As a result there have been various attempts to vectorize the persistence diagram, in order that the summary is more amenable to statistical analysis and machine learning techniques.

The article \cite{Bubenik:2015} introduces a stable vectorization of the persistence diagram, the persistence landscape, a function in Lebesgue $p$-space. This summary naturally enjoys unique means and one can perform traditional hypothesis tests upon this summary and numerical statistics derived from this summary.

There are several natural situations where we may wish to build a richer structure of topological spaces on a data set and track the changes to the homology whilst varying multiple parameters. For example in \cite{Keller2018} the topology of chemical compounds is studied using 2-parameter filtrations. The PH theory becomes wildly more complicated when the family of topological spaces on our data set is indexed by multiple parameters. The associated family of homology groups is known as a multiparameter persistence module.  The theory of multiparameter persistence modules is presented in \cite{Carlsson2009} and unlike the single parameter case where we may associate a persistence diagram to a module, there is not an analogous complete discrete invariant in the multiparameter setting.

There exist various approaches to define invariants for multiparameter persistence modules in the literature.
The rank invariant of a module has been studied in the context of $H_0$-modules and shape matching, and has been shown to be stable when endowed with a matching distance \cite{frosini_persistent_2011} \cite{cerri_betti_2013}. An alternative approach uses algebraic geometry to construct numeric functions on multiparameter persistence modules \cite{Skryzalin2017}, generalising the ring of algebraic functions on barcode space for the single parameter case \cite{Adcock2013}. However, the disadvantage of both of these approaches is that equipped with their vector space norms, these invariants are unstable with respect to the natural distance to consider on multiparameter modules, the interleaving distance.
Another approach is to study algebraic invariants associated to the multigraded algebra structure of multiparameter persistence modules \cite{2017Otter}.

In this article we introduce a family of new stable invariants for multiparameter persistence modules, naturally extending the results of \cite{Bubenik:2015} from the setting of single parameter persistence modules to multiparameter persistence modules. Our incomplete invariants, the \textit{multiparameter persistence landscapes}, are derived from the rank invariant associated to a multiparameter persistence module, and are continuous functions in Lesbegue $p$-space. As such, they are naturally endowed with a distance function and are well suited to statistical analysis, since again there is a uniquely defined mean associated to multiple landscapes. The natural inner-product structure on the landscape functions gives rise to a positive-definite kernel, which can be leveraged by machine learning algorithms.

The multiparameter landscape functions are sensitive to homological features of large, medium and small persistence. The landscapes also have the advantage of being interpretable since they are closely related to the rank invariant. Moreover one can derive stable $\mathbb{R}$-valued numeric invariants from the landscape functions using the linear functionals in the dual space. We can produce confidence intervals and perform hypothesis tests on these numeric invariants which are viewed as $\mathbb{R}$-valued random variables.

In the $2$-parameter case we visualise the multiparameter landscapes as a surface $\lambda : \mathbb{R}^2 \to \mathbb{R}$. We shall present computational examples in the $2$-parameter case using the RIVET software presented in \cite{Lesnick2015}. These examples will serve a range of potential applications, demonstrating that the landscapes are sensitive to both the topology and geometry of a data set.

\subsection{Outline of Content}

We begin by introducing multiparameter persistence theory in Section \ref{Multiparameter-Theory}. We shall present the interleaving distance on multiparameter modules and define the persistence weighted Wasserstein distance for interval decomposable modules.

In Section \ref{Landscapes} we will recall the definition and properties of the single parameter persistence landscape. We shall consider possible generalisations of the persistence landscape to multiparameter persistence modules and then explore the properties and discriminating power of our chosen generalisation.

Section \ref{Stability} contains the proof of stability of our multiparameter persistence landscapes with respect to the interleaving distance and persistence weighted Wasserstein distance. We include also an optimality result showing that the collection of landscapes associated to a multiparameter persistence module contains almost all the information contained in the rank invariant.

In Section \ref{Statistics} we shall collect results from the literature regarding Banach Space valued random variables and demonstrate how these results apply to multiparameter persistence landscapes.

Finally, in Section \ref{Computations} we provide a conservative estimate for the time complexity of computing multiparameter persistence landscapes, perform some example computations of the multiparameter landscapes and apply statistical tests to various data sets. 



\section*{Acknowledgements}

The author would like to give recognition to the \textit{Theory and Foundations of TGDA} workshop hosted at Ohio State University which facilitated useful conversations with experts in TGDA. The author wishes to thank his supervisor Ulrike Tillmann for her guidance and support with this project, and Peter Bubenik for helpful suggestions. The author gratefully acknowledges support from EPSRC studentship EP/N509711/1 and EPSRC grant EP/R018472/1.

\section{Multiparameter Persistence Theory}
\label{Multiparameter-Theory}
The theory of persistent homology is well developed for topological spaces filtered over a single parameter. Under appropriate finiteness conditions, the Krull-Schmidt Decomposition Theorem establishes the barcode as a complete invariant \cite{Crawley-Boevey15}. 

In the more general situation we study our space is filtered over multiple parameters. What we call multiparameter persistence is sometimes termed multidimensional persistence in the literature. We shall reserve \textit{parameter} to describe variables over which we are filtering our space, and \textit{dimension} to refer to homological dimension. 

The general decomposition theorem available in the single parameter case does not generalise to multiparameter modules. In the language of Quiver representations, single parameter persistence is the study of $A_n$-type quiver representations and their infinite extensions, whilst in contrast multiparameter persistence concerns the more complicated representations of wild-type quivers. However recent work has shown that affording additional structure to poset modules results in an appropriate generalisation of Gabriel's Theorem under certain conditions \cite{Ogle2018}.  

Let us begin with an exposition of multiparameter persistence theory. We shall carry two equivalent perspectives of multiparameter persistence modules; a categorical perspective which will serve efficient descriptions of interleavings, and a module theoretic perspective which will serve an efficient way to describe modules via presentations.

The following example gives a construction of a multiparameter persistence module.

\begin{ex}(Sublevel-set Multiparameter Persistence Module)

  Let $X$ be a topological space and $f : X \to \mathbb{R}^n$ a filtering function. We can associate a family of topological subspaces indexed by vectors $\vec{a} = (a_1,...,a_n) \in \mathbb{R}^n$ induced by $f$: $$X_\vec{a} = \{ x\in X\ |\ f(x)_i < a_i \ \forall \  i = 1,...,n\}$$
  this is known as the sublevel set filtration. 
  For any $\vec{b}\in \mathbb{R}^n$ such that $a_i \leq b_i$ for all $i = 1,..,n$, we have an inclusion map $X_\vec{a} \hookrightarrow X_\vec{b}$. If we let $H$ denote a singular homology functor with coefficients in a field then applying this functor to the collection $\{X_\vec{a}\}$ and the appropriate inclusion maps gives rise to a family of vector spaces and linear maps known as a Sublevel-set Multiparameter Persistence Module.
\end{ex}

\subsection{Multiparameter Persistence Modules}

Let $P_n$ denote the monoid ring of the monoid $([0,\infty)^n,+)$ over a field $\mathbb{F}$. Equivalently one may think of $P_n$ as a pseudo-polynomial ring $\mathbb{F}[x_1,...,x_n]$ in which exponents are only required to be non-negative and can be non-integral. Let $A_n$ denote the polynomial ring $\mathbb{F}[x_1,...,x_n]$ or analogously the monoid ring of $(\mathbb{N}^n,+)$ over $\mathbb{F}$.

Let $\vec{P}$ denote the category associated to the poset $P$, so that  $\vec{R}^n$ and $\vec{Z}^n$ denote the categories associated to the posets $(\mathbb{R}^n,\leq)$ and $(\mathbb{Z}^n,\leq)$ under the standard coordinate-wise partial orders. Let $\vec{Vect}$ denote the category of vector spaces and linear maps over $\mathbb{F}$, and $\vec{vect}$ denote the subcategory of finite dimensional vector spaces. Moreover let us denote the $\textbf{C}$ valued functors on $\textbf{P}$ by $\textbf{C}^\textbf{P}$.

\begin{defn}(Persistence Module)
Let $M$ be a module over the ring $P_n$. We say $M$ is a persistence module if $M$ is an $\mathbb{R}^n$-graded $P_n$-module. That is to say $M$ has a decomposition as a $\mathbb{F}$-vector space $M = \bigoplus_{\vec{a}\in \mathbb{R}^n} M_\vec{a}$ compatible with the action of $P_n$: 
$$m \in M_\vec{a} \implies \vec{x}^\vec{b} \cdot m \in M_{\vec{a}+\vec{b}}$$
Recall that we require morphisms of graded modules to respect grading and be compatible with the module structure.
\end{defn}

In the setting of a sublevel-set persistence module the vector space at each grade is $H(X_\vec{a})$, and the action of $\vec{x}^\vec{b}$ on $H(X_\vec{a})$ is given by the linear map of homology groups induced by the inclusion $X_\vec{a} \hookrightarrow X_{\vec{a}+\vec{b}}$. 

\begin{defn}(Persistence Module)
Let $M$ be an element of the functor category $\vec{Vect}^{\vec{R}^n}$ then $M$ is a persistence module. A morphism of persistence modules is simply a natural transformation $M \Rightarrow M'$.
\end{defn}

\begin{defn}(Discrete Persistence Module)
Let $M$ be a module over the ring $A_n$. We say $M$ is a discrete persistence module if $M$ is an $\mathbb{Z}^n$-graded $A_n$-module.
\end{defn}

\begin{defn}(Discrete Persistence Module)
Let $M$ be an element of the functor category $\vec{Vect}^{\vec{Z}^n}$ then $M$ is a discrete persistence module.
\end{defn}

The equivalence of the two perspectives is simply saying that we have an equivalence of categories between $\mathbb{R}^n$-graded $P_n$-$\vec{Mod}$ and $\vec{Vect}^{\vec{R}^n}$. 

We have introduced the notion of discrete persistence modules since the computations we present shall deal with discrete persistence modules approximating continuous persistence modules. This approximation can be taken to an arbitrary degree of accuracy (under the interleaving distance) with increasing computational cost.

Discretization is not the only approach one can take in computations involving continuous modules. In contrast \cite{Miller2017} develops a primary decomposition of modules which facilitates a finite description of a wide class of persistence modules which would require infinitely many generators if discretized. Our only obstruction to using this approach rather than discretization is the lack of available software to cope with these presentations.

Associated to a multiparameter module is a family of single parameter modules whose collection of barcodes is known as the fibered barcode space.

\begin{defn}(Fibered Barcode Space)
Let $\vec{L}$ denote the subposet of $\vec{R}^n$ corresponding to a positively sloped line $L\subset \mathbb{R}^n$. Let $\iota_L : (\vec{R},\|\cdot\|_\infty) \to (\vec{R}^n,\|\cdot\|_\infty)$ denote the isometric embedding with $\iota_L( \vec{R})=\vec{L}$ and $\iota_L(0)\in \{x_n = 0\}$. Then for $M\in\vec{Vect}^{\vec{R}^n}$ the composite $M^L = M\circ \iota_L$ is a single parameter persistence module, and thus has an associated barcode $\mathcal{B}(M^L)$. Let $\mathcal{L}$ denote the set of positively sloped lines then the collection $\{\mathcal{B}(M^L) \ :\ L \in \mathcal{L} \}$ is known as the fibered barcode space of $M$.
\end{defn}

We shall see later that we will be able to reduce the computation of certain multiparameter landscapes to queries of the fibered barcode space.

\subsection{Presentations}

Let us now develop the theory required to define presentations of persistence modules. 

\begin{defn}(Translation Endofunctors)\cite{Bubenik2015}
Let $\vec{P}$ be the category associated to a preordered set (proset) and let $\Gamma : \vec{P} \to \vec{P}$ be an endofunctor. We say that $\Gamma$ is a translation. Since $\Gamma$ is a functor $\Gamma$ is monotone $x\leq y \implies \Gamma(x)\leq \Gamma(y)$. We say that $\Gamma$ is increasing if $x \leq \Gamma(x)$ for all $x \in P$. 

Let $\textbf{Trans}_{\textbf{P}}$ denote the set of increasing translations of $\vec{P}$ and observe that $\textbf{Trans}_{\textbf{P}}$ is a monoid with respect to composition. 
\end{defn}

It is straight forward to see that $\textbf{Trans}_{\textbf{P}}$ also has a natural proset structure with preorder $\Gamma \leq K \Leftrightarrow \Gamma(x) \leq K(x) \text{ for all } x$ . This preorder is compatible with the monoid structure and $\Gamma \leq K$ implies there is a unique natural transformation $\eta_{K}^{\Gamma} : \Gamma \Rightarrow K $. If $\textbf{P}$ is a poset then so is $\textbf{Trans}_{\textbf{P}}$.

\begin{defn}(Shift Functor)
Let $F$ be an element of the functor category $\vec{C}^\vec{P}$ and $\Gamma$ a translation endofunctor. Let $F(\Gamma)$ denote $F\circ \Gamma \in \vec{C}^\vec{P}$, we call this functor the $\Gamma$ shift of $F$.
\end{defn}

For a persistence module $M$ we shall write $M(\vec{a})$ to denote the shift by the translation in $\mathbb{R}^n$, $\Gamma_{\vec{a}}(\vec{x})= \vec{x} + \vec{a}$.
We define a graded set to be some subset $\mathcal{X} \subset J \times P$ where $J$ is an arbitrary indexing set and $P$ is a grading set. For an element of a graded set $(j,\vec{a})$, we shall refer to $\vec{a}$ as the grade of $j$, gr$(j) = \vec{a}$.

\begin{defn}(Free Module)
Let $\mathcal{X}$ be an $\mathbb{R}^n$-graded set we define the free module on $\mathcal{X}$ to be:
$$ \text{Free}[\mathcal{X}] = \bigoplus_{j\in J} P_n(-\text{gr}(j))$$
\end{defn}

The notion of a free module on a graded set can equivalently be defined using a univeral property characterisation.

We say a subset $\mathcal{R} \subset M$, of a persistence module is homogeneous if $\mathcal{R} \subset \cup_{\vec{a}\in \mathbb{R}^n} M_\vec{a}$, that is to say each element has a well defined grade.

\begin{defn}(Presentations)
Let $\mathcal{X}$ be a graded set and $\mathcal{R}$ a homogeneous subset of the free module on $\mathcal{X}$ generating the submodule $\langle \mathcal{R} \rangle$. We say that a persistence module $M$ has presentation $\langle \mathcal{X} | \mathcal{R}\rangle$ if:
$$ M \cong \frac{\text{Free}[\mathcal{X}]}{\langle \mathcal{R} \rangle}$$
We say that a presentation is finite if both $\mathcal{X}$ and $\mathcal{R}$ are finite. Let $I$ denote the ideal of $P_n$ generated by the elements $\{\vec{x}^\vec{a} \ | \ \vec{a} > 0 \}$ and let $\Phi_{\langle \mathcal{X} | \mathcal{R}\rangle} : \text{Free}[\mathcal{R}] \to \text{Free}[\mathcal{X}]$ be the map induced by the inclusion $\mathcal{R} \hookrightarrow \text{Free}[\mathcal{X}]$. We say that a presentation of $M$ is minimal if $\mathcal{R} \subset I \cdot \text{Free}[\mathcal{X}]$ and $\ker_{\Phi_{\langle \mathcal{X} | \mathcal{R}\rangle}} \subset I \cdot \text{Free}[\mathcal{R}]$.
\end{defn}

\begin{defn}(Multiparameter Betti Numbers)
Let $M$ be a persistence module, then the associated Betti numbers are maps $\xi_i(M) : \mathbb{R}^n \to \mathbb{N}$ defined by:
$$ \xi_i(M)(\vec{a}) = \dim_{\mathbb{F}} (\text{Tor}^{P_n}_i(M,P_n / I P_n)_\vec{a})$$
Standard homological algebra arguments establish that the Betti numbers are well defined (see \cite{Lesnick2015} for details).
If $\langle \mathcal{X} | \mathcal{R}\rangle$ is a minimal presentation for $M$ then $\xi_0(M)(\vec{a}) = |\text{gr}^{-1}_\mathcal{X}(\vec{a})|$ and $\xi_1(M)(\vec{a}) = |\text{gr}^{-1}_\mathcal{R}(\vec{a})|$
\end{defn}

The multiparameter Betti numbers are related to the initial topological space with $\xi_0(M)$ marking the filtration values for the birth of homological features, and $\xi_1(M)$ marking the filtration values for relations between features.

\subsection{Generalised Interleavings}

We shall now adopt the notion of a generalised interleaving from \cite{Bubenik2015} and define an interleaving distance on multiparameter persistence modules. For a more detailed account of interleavings of multiparameter persistence modules see also \cite{Lesnick2012}. 

\begin{defn}(Interleaving) \cite{Bubenik2015}
Let $F,G \in \vec{C}^\vec{P}$ be modules and $\Gamma, K \in \textbf{Trans}_{\textbf{P}}$. We say that $F,G$ are $(\Gamma, K)$-interleaved if there exist natural transformations $\varphi : F \Rightarrow G\Gamma$, $\psi: G\Rightarrow FK$ satisfying the coherence criteria that $(\psi \Gamma)\varphi  = F \eta^{\id}_{K\Gamma} $, $(\varphi K)\psi  = G \eta^{\id}_{\Gamma K}$ where $\eta^{\id}_{\alpha}$ denotes the unique natural transformation between the translations $\id \leq \alpha$.
\end{defn}

An interleaving may be thought of as an approximate isomorphism. Indeed if we take $\Gamma = K = \id$ then $F,G$ are $(\Gamma, K)$-interleaved if and only if $F,G$ are isomorphic. By warping the poset with translations $\Gamma,K$ we admit flexibility to the rigid notion of isomorphism. In order to introduce an associated distance we must assign a weight to the translations to quantify how close the interleaving is to an isomorphism.

\begin{defn}(Sublinear Projections and $\varepsilon$-Interleavings) \cite{Bubenik2015}

A sublinear projection is a function $ \omega : \textbf{Trans}_{\textbf{P}} \rightarrow [0,\infty]$ such that $\omega_{\id}=0$ and $\omega_{\Gamma K} \leq \omega_{\Gamma} + \omega_{K}$. We say that a translation $\Gamma$ is an $\varepsilon$-translation if $\omega_{\Gamma} \leq \varepsilon$. We say modules $F,G \in \textbf{C}^\textbf{P}$ are $\varepsilon$-interleaved with respect to $\omega$ if they are $(\Gamma,K)$-interleaved for some pair of $\varepsilon$-translations.
\end{defn}

\begin{prop}(Induced Interleaving Distance) \cite{Bubenik2015}

Given a sublinear projection $\omega$ and modules $F,G \in \textbf{C}^\textbf{P}$, we have an induced interleaving distance given by:
 $$ d^{\omega}(F,G) = \inf \{ \varepsilon \geq 0 : F,G \text{ are $\varepsilon$-interleaved with respect to $\omega$} \}$$
\end{prop}

A common sublinear projection to consider is given by $\omega_{\Gamma} = \| \Gamma - \id \|_\infty$. We may sometimes refer to the interleaving distance induced by this sublinear projection as simply \textit{the interleaving distance}. We will specify the sublinear projection when we wish to consider an alternative interleaving distance.

The dual notion to a sublinear projection is a \textit{superlinear family} from which one can also derive an interleaving distance.

\begin{defn}(Superlinear Family) \cite{Bubenik2015}
Let $\Omega : [ 0,\infty ) \to \vec{Trans}_{\vec{P}}$ be a superlinear function: $\Omega_{\varepsilon_1+ \varepsilon_2} \geq \Omega_{\varepsilon_1}\Omega_{\varepsilon_2}$. Then we say that $\Omega$ is a superlinear family.
\end{defn}

\begin{prop}(Induced Interleaving Distance) \cite{Bubenik2015}

Given a superlinear family $\Omega$ and modules $F,G \in \textbf{C}^\textbf{P}$, we have an induced interleaving distance given by:
 $$ d^{\Omega}(F,G) = \inf \{ \varepsilon \geq 0 : F,G \text{ are $\Omega_\varepsilon$-interleaved} \}$$
\end{prop}

The interleaving distance has been shown to be NP-hard to compute \cite{Botnan17}. Nevertheless the interleaving distance is a very natural distance to consider on persistence modules. \cite{Lesnick2012} establishes that the interleaving distance is universal amongst stable distances on persistence modules, that is to say any other stable distance is bounded above by the interleaving distance. This property provides a strong justification for considering the interleaving distance on multiparameter modules.

\subsection{Discretization and Continuous Extension}

In Section \ref{Computations} our computations will be simplified by restricting a continuous module to a finite grid and then dealing with the continuous extension of this discretization. We will show that restricting to a finite grid gives us a suitable approximation to our module with respect to the interleaving distance between modules.

\begin{defn}(Grid Function)
Let $\mathcal{G}: \mathbb{Z}^n \to \mathbb{R}^n$ be defined by component-wise increasing functions $\mathcal{G}_i : \mathbb{Z} \to \mathbb{R}$ with $\sup \mathcal{G}_i = \sup -\mathcal{G}_i = \infty$. Then we say $\mathcal{G}$ is a grid function. Let us define the size of $\mathcal{G}$ to be $$|\mathcal{G}| = \max_{i\in[n]}\sup_{z \in \mathbb{Z}}|\mathcal{G}_i(z) - \mathcal{G}_i(z+1)|$$
\end{defn}

\begin{defn}(Discretization)
Let $M\in \vec{Vect}^{\vec{R}^n}$ be a persistence module and $\mathcal{G}$ a grid function. We say that $M\circ \mathcal{G} \in \vec{Vect}^{\vec{Z}^n}$ is the $\mathcal{G}$-discretization of $M$.
\end{defn}

\begin{defn}(Continuous Extension)\cite{Lesnick2015}
Let $Q\in \vec{Vect}^{\vec{Z}^n}$ be a discrete persistence module and $\mathcal{G}$ a grid function. For $\vec{x}\in \mathbb{R}^n$ let us define floor and ceiling functions: 
$$\lfloor \vec{x} \rfloor_\mathcal{G} = \max\{\vec{z}\in \Ima \mathcal{G} \ |\ \vec{z} \leq \vec{x}\} \text{ and } \lceil \vec{x} \rceil_\mathcal{G} = \min\{\vec{z}\in \Ima \mathcal{G} \ |\ \vec{z} \geq \vec{x}\}$$ 

We define the continuous extension $E_{\mathcal{G}}(Q)\in\vec{Vect}^{\vec{R}^n}$ to be the persistence module with:
$$E_{\mathcal{G}}(Q)_{\vec{a}} = Q_{\mathcal{G}^{-1}(\lfloor \vec{a} \rfloor_\mathcal{G})} \text{ and } E_{\mathcal{G}}(Q)(\vec{a} \leq \vec{b}) = Q(\mathcal{G}^{-1}(\lfloor \vec{a} \rfloor_\mathcal{G}) \leq \mathcal{G}^{-1}(\lfloor \vec{b} \rfloor_\mathcal{G}))$$

With the obvious action on the morphisms of $\vec{Vect}^{\vec{Z}^n}$, we have defined a functor $ E_{\mathcal{G}}: \vec{Vect}^{\vec{Z}^n} \to \vec{Vect}^{\vec{R}^n}$.
\end{defn}

The following proposition shows that discretization is stable with respect to the interleaving distance, and so we may produce an arbitrarily close approximation to a persistence module by restricting the module to a grid of sufficiently small size.

\begin{prop}
Let $M \in \vec{Vect}^{\vec{R}^n}$ be a persistence module, $\mathcal{G}$ a grid function and $\omega$ the sublinear projection given by $\omega_{\Gamma} = \| \Gamma - \id \|_\infty$ then we have that:
$$ d^{\omega}(M, E_{\mathcal{G}}(M \circ \mathcal{G})) \leq | \mathcal{G} | $$
\end{prop}
\begin{proof}
The modules $M, E_{\mathcal{G}}(M \circ \mathcal{G})$ are $(\lceil \cdot \rceil_{\mathcal{G}} ,\id)$-interleaved with natural transformations given by the appropriate internal morphisms of $M$.
\end{proof}

\subsection{Interval Decomposable Modules}

Given the complicated nature of unconstrained persistence modules, it is common to consider subclasses of multiparameter persistence modules.

\begin{defn}(Interval Decomposable Modules)

Let $P$ be a poset. We define a subposet $I\leq P$ to be an interval if $s,t\in I, s\leq r\leq t \implies r\in I$ and for any $s,t\in I \  \exists \  r_i \in I$ connecting $s$ and $t$, $s =r_0 \leq r_1 \geq r_2 \leq r_3 \geq ... \leq r_n = t$. The interval module $\mathds{1}^{\vec{I}} \in \vec{vect}^{\vec{P}}$ associated to an interval $I$ has a one dimensional vector space at each $a\in I$ and internal isomorphisms given by the identity wherever possible. We say a module $M \in \vec{vect}^{\vec{P}}$ is interval decomposable if $M \cong \bigoplus_{j \in \mathcal{J}} \mathds{1}^{\vec{I}_j}$, for some multiset of intervals $\{\{I_j\}\}$.

\end{defn}

It is straight forward to verify that the endomorphism ring of an interval summand is given by the field over which we are working. In particular this ring is local and so the Krull-Schmidt-Remak-Azumaya Theorem guarantees that the decomposition of an interval decomposable module is unique up to reordering. We can thus assign the multiset of intervals in the decomposition of a module $M$ to be the barcode $\mathcal{B}(M) = \{\{I_j\}\}$.



Restricting our attention to interval decomposable modules we can see the complicated nature of interleavings of multiparameter modules. The class of interval decomposable modules admit a bottleneck distance. 

\begin{defn}($\varepsilon$-Matching)

Let $\{I_j\ \ | \ j \in \mathcal{J}\}$ and $\{J_k\ \ | \ k \in \mathcal{K}\}$ be multisets of intervals. We say a partial bijection $\sigma : \mathcal{J} \nrightarrow \mathcal{K}$ is an $\varepsilon$-matching if $d^{\omega}(\mathds{1}^{I_j},\mathds{1}^{J_{\sigma(j)}}) \leq \varepsilon$ for matched intervals and $d^{\omega}(\mathds{1}^{I_j},0),d^{\omega}(0,\mathds{1}^{J_{k}}) \leq \varepsilon$ for unmatched intervals.

\end{defn}

\begin{defn}(Bottleneck Distance)

Let $M$ and $N$ be interval decomposable modules. The bottleneck distance between the modules is given by:

$$ d_B(M,N) = \inf \{\varepsilon \geq 0 \ |\ \mathcal{B}(M),\mathcal{B}(N) \text{ admit an } \varepsilon \text{-matching} \}$$

\end{defn}

One would hope to attain a result analogous to the isometry theorem for ordinary one dimensional persistent homology relating the bottleneck distance and the interleaving distance. In the single parameter case an $\varepsilon$-interleaving induces an $\varepsilon$-matching between summands \cite{bauer_induced_2014}. In contrast, the interleaving distance and bottleneck distance do not coincide for multiparameter interval decomposable modules. Certainly the bottleneck distance provides an upper bound on the interleaving distance. However general interleavings of interval decomposable multiparameter modules do not necessarily induce a matching of interval summands. This is best illustrated by the example provided in \cite{Bjerkevik2016} for which the optimal matching between 1-interleaved modules is a 3-matching, see Figure \ref{BottleneckVsInterleaving}. 

We can further define a Wasserstein distance for interval decomposable modules.

\begin{defn}($p$-Wasserstein Distance)
Let $M,N$ be interval decomposable persistence modules with barcodes  $\{I_j\ \ | \ j \in \mathcal{J}\}$ and $\{J_\kappa \ \ | \ \kappa \in \mathcal{K}\}$ respectively. Assume the cardinality of these barcodes coincide, by appending a collection of empty intervals to each barcode. For a matching $\sigma : \mathcal{J} \to \mathcal{K}$ let $\varepsilon_j = d^{\omega}(\mathds{1}^{I_j},\mathds{1}^{J_{\sigma(j)}})$. The $p$-Wasserstein distance is given by:

$$d_{W_p}(M,N) = \inf_{\sigma : \mathcal{J} \to \mathcal{K} } \left[\sum_\mathcal{J} \varepsilon_j^p \right]^{\frac{1}{p}}$$

\end{defn}

The bottleneck distance is simply the $\infty$-Wasserstein distance. If we wish to place extra emphasis on intervals with large persistence we may use the persistence weighted $p$-Wasserstein distance.

\begin{defn}(Persistence Weighted $p$-Wasserstein Distance)
Let $M,N$ be interval decomposable persistence modules with barcodes  $\{I_j\ \ | \ j \in \mathcal{J}\}$ and $\{J_\kappa \ \ | \ \kappa \in \mathcal{K}\}$. For a subset $A\subset \mathbb{R}^n$ let $|A|$ denote the Euclidean volume. The persistence weighted $p$-Wasserstein distance is given by:

$$d_{\overline{W}_p}(M,N) = \inf_{\sigma : \mathcal{J} \to \mathcal{K} } \left[\sum_\mathcal{J} |I_j \cup J_{\sigma(j)}|\varepsilon_j^p \right]^{\frac{1}{p}}$$

\end{defn}

The landscape distance we introduce in the following section is similar to the persistence weighted $p$-Wasserstein Distance, and can be defined for persistence modules which do not admit an interval decomposition. 

In Section \ref{Stability} we will show that our invariant is stable with respect to interleaving distance and the persistence weighted Wasserstein distance. In particular the distance function on multiparameter landscapes provides a lower bound on interleaving distance. Given the NP-hardness of computing interleaving distance this lower bound may prove useful for comparing multiparameter persistence modules \cite{Botnan17}.

\begin{figure}
\centering
\begin{subfigure}[b]{0.3\linewidth}
\begin{tikzpicture}[scale = 0.7]

\filldraw[fill=green, fill opacity = 0.25] (0,0.5) rectangle (5,5.5);
\filldraw[fill=green, fill opacity = 0.25] (0,-.5) rectangle (6,5.5);
\filldraw[fill=green, fill opacity = 0.25] (1,.5) rectangle (5,4.5);
\node [anchor=south] (I) at (3,-1.5) {$M$};
\end{tikzpicture}
\end{subfigure}
\begin{subfigure}[b]{0.3\linewidth}
\begin{tikzpicture}[scale = 0.7]
\filldraw[fill=red, fill opacity = 0.25] (0.5,0) rectangle (5.5,5);
\filldraw[fill=red, fill opacity = 0.25] (0.5,0) rectangle (4.5,6);
\filldraw[fill=red, fill opacity = 0.25] (-.5,1) rectangle (5.5,5);
\node [anchor=south] (I) at (3,-1.5) {$N$};
\end{tikzpicture}
\end{subfigure}
\begin{subfigure}[b]{0.3\linewidth}
\begin{tikzpicture}[scale = 0.70]
\filldraw[fill=green, fill opacity = 0.25] (0,0.5) rectangle (5,5.5);
\filldraw[fill=green, fill opacity = 0.25] (0,-.5) rectangle (6,5.5);
\filldraw[fill=green, fill opacity = 0.25] (1,.5) rectangle (5,4.5);
\filldraw[fill=red, fill opacity = 0.25] (0.5,0) rectangle (5.5,5);
\filldraw[fill=red, fill opacity = 0.25] (0.5,0) rectangle (4.5,6);
\filldraw[fill=red, fill opacity = 0.25] (-.5,1) rectangle (5.5,5);
\node [anchor=south] (I) at (3,-1.5) {$M$ and $N$ overlayed};
\end{tikzpicture}
\end{subfigure}
\caption{We illustrate rectangular decomposable modules $M =(0,10)\times (1,11) \oplus (0,12)\times(-1,11)\oplus (2,10)\times (1,9)$ and $N =(1,11)\times (0,10) \oplus (1,9)\times(0,12)\oplus (-1,11)\times (2,10)$. There is a $1$-interleaving between $M,N$ but there is no $1$-matching. The optimal matching is a $3$-matching. See \cite{Bjerkevik2016} for details.}
\label{BottleneckVsInterleaving}
\end{figure}
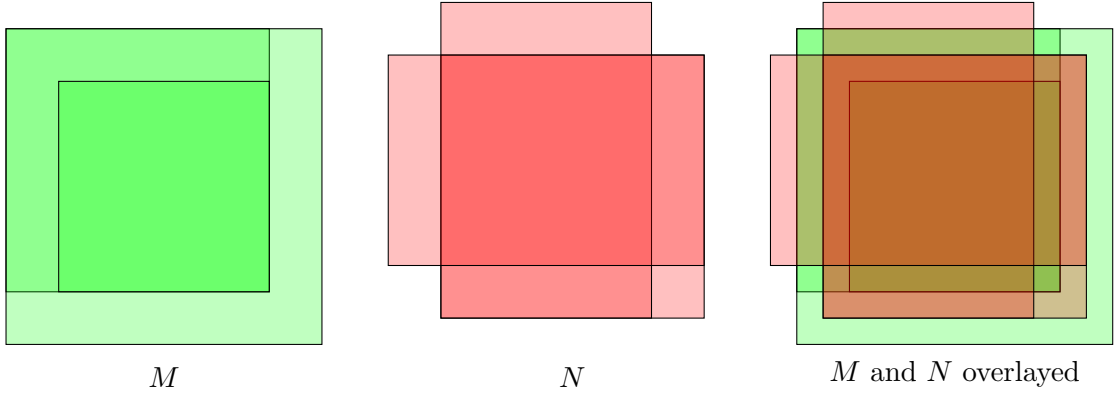

\section{Persistence Landscapes}
\label{Landscapes}

In this section we shall recall the definition of the single parameter persistence landscape and its properties. We shall generalise the definition to multiparameter persistence modules and show which properties of the single parameter persistence landscape are preserved.

From this point onward all single parameter persistence modules we consider shall be pointwise finite dimensional and satisfy the finiteness conditions of \cite{Crawley-Boevey15} in order that they admit an interval decomposition. The multiparameter persistence modules we consider will be pointwise finite dimensional and tame, but will not necessarily admit an interval decomposition.

\subsection{Single Parameter Persistence Landscapes}
The \textit{persistence landscape} associated to a single parameter persistence module is defined in \cite{Bubenik:2015}. The persistence landscape is derived from the rank invariant of a module.

\begin{defn}(Rank Invariant)
Let $M \in \textbf{vect}^{\textbf{R}}$ be a persistence module then  for $a\leq b$ the function $\beta^{\cdot,\cdot}$ giving the corresponding Betti number is the rank invariant of $M$:
$$ \beta^{a,b} = \dim(\Ima(M(a\leq b)))$$
\end{defn}

\begin{defn}(Rank Function)
The rank function $\rk : \mathbb{R}^2 \rightarrow \mathbb{R}$ is given by
$$\rk(b,d)= \begin{cases} \beta^{b,d} & \text{if } b \leq d \\ 0 & \text{otherwise} \end{cases} $$
\end{defn}

\begin{defn}(Rescaled Rank Function)
The rescaled rank function $r : \mathbb{R}^2 \rightarrow \mathbb{R}$ is supported on the upper half plane:
$$r(m,h)= \begin{cases} \beta^{m-h,m+h} & \text{if } h \geq 0\\
0 & \text{otherwise} \end{cases} $$
\end{defn}

Observe that the rank function has support contained in the upper triangular half of the plane with the coordinates corresponding to ``births" and ``deaths", whilst the rescaled rank function has support contained in the upper half plane with coordinates corresponding to ``midpoints" and ``half-lifes".

\begin{defn}(Persistence Landscape)
The persistence landscape is a function $\lambda : \mathbb{N} \times \mathbb{R} \rightarrow \overline{\mathbb{R}}$,
where $\overline{\mathbb{R}}$ denotes the extended real numbers, $[-\infty,\infty]$. Alternatively, it may be thought of as a sequence of
functions  $\lambda_{k} : \mathbb{R} \rightarrow \overline{\mathbb{R}}$, where  $\lambda_{k}(t) = \lambda(k,t)$. Define
$$ \lambda(k,t) = \sup\{ h \geq 0 : \beta^{t-h,t+h} \geq k\}$$
\end{defn}

The value $\lambda(k,t)$ gives the maximal radius of an interval centred at $t$ that is contained in at least $k$ intervals of the barcode. The persistence landscape and persistence diagram of a suitably well-behaved single parameter module carry the same information.

Alternatively the persistence landscape of a single parameter module can be derived from the landscape functions of the modules interval summands, see Figure \ref{fig:Persistence-Landscape-Example}.

\begin{defn}(Persistence Landscape)
    Let $M$ be a single parameter persistence module with associated persistence diagram given by the multiset $\{\{(b_j,d_j)\}\}$. The persistence landscape may be equivalently defined as:
    $$  \lambda_M(k,t) = \text{kmax}\{ \lambda_{(b_j,d_j)}(1,t) \} $$
    Where \text{kmax} denotes the $k^{\text{th}}$ largest value of the multiset and $\lambda_{(b_j,d_j)}$ is the landscape associated to the interval module  $\mathds{1}^{(b_j,d_j)}$.
\end{defn}

We shall later see that our chosen multiparameter generalisation of the persistence landscape also admits a decomposition as the $\text{kmax}$ of a series of simple landscape functions when our module is interval decomposable.

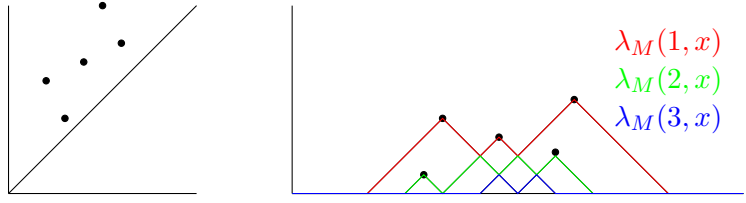
\begin{figure}
\begin{center}
\begin{tikzpicture}

\draw (0,0) -- (2.5,0);
\draw (0,0) -- (0,2.5);
\draw (0,0) -- (2.5,2.5);

\node at (0.5,1.5) [draw,circle,fill = black, inner sep = 0.3mm]{};
\node at (.75,1) [draw,circle,fill = black, inner sep = 0.3mm]{};
\node at (1,1.75) [draw,circle,fill = black, inner sep = 0.3mm]{};
\node at (1.5,2) [draw,circle,fill = black, inner sep = 0.3mm]{};
\node at (1.25,2.5) [draw,circle,fill = black, inner sep = 0.3mm]{};
\end{tikzpicture}
\hspace{1cm}
\begin{tikzpicture}

\draw (0,0) -- (6,0);
\draw (0,0) -- (0,2.5);

\draw (1,0) -- (2,1) -- (3,0);
\draw (1.5,0) -- (1.75,0.25) -- (2,0);
\draw (2,0) -- (2.75,0.75) -- (3.5,0);
\draw (3,0) -- (3.5,0.5) -- (4,0);
\draw (2.5,0) -- (3.75,1.25) -- (5,0);

\node at (2,1) [draw,circle,fill = black, inner sep = 0.3mm]{};
\node at (1.75,0.25) [draw,circle,fill = black, inner sep = 0.3mm]{};
\node at (2.75,0.75) [draw,circle,fill = black, inner sep = 0.3mm]{};
\node at (3.5,0.55) [draw,circle,fill = black, inner sep = 0.3mm]{};
\node at (3.75,1.25) [draw,circle,fill = black, inner sep = 0.3mm]{};

\node[align=right] at (5,2) {\textcolor{red}{$\lambda_M(1,x)$}};
\node[align=right] at (5,1.5) {\textcolor{green}{$\lambda_M(2,x)$}};
\node[align=right] at (5,1) {\textcolor{blue}{$\lambda_M(3,x)$}};

\draw[red] (0,0) -- (1,0) -- (2,1) -- (2.5,0.5) -- (2.75,0.75) -- (3,0.5) -- (3.75,1.25) -- (5,0) -- (6,0);
\draw[green] (0,0) -- (1.5,0) -- (1.75,0.25) -- (2,0) -- (2.5,0.5) -- (2.75,0.25) -- (3,0.5) -- (3.25,0.25) -- (3.5,0.5) -- (4,0) -- (6,0);
\draw[blue] (0,0) -- (2.5,0) -- (2.75,0.25) -- (3,0) -- (3.25,0.25) -- (3.5,0) -- (6,0);
\end{tikzpicture}
\end{center}
\caption{We show the persistence diagram of a single parameter module $M$ on the left and the associated persistence landscapes on the right. One can see that $\lambda_M(k,t)$ is the $\text{kmax}$ of the landscape functions $\lambda_{(b_j,d_j)}(1,t)$ of the interval summands of $M$.}
\label{fig:Persistence-Landscape-Example}
\end{figure}

\begin{lem}
The persistence landscape has the following properties:

\begin{enumerate}

\item $\lambda_{k}(t) \geq 0$
\item $\lambda_{k}(t) \geq \lambda_{k+1}(t)$
\item  $\lambda_{k}(t) \text{ is 1-Lipschitz}$
 
\end{enumerate}
The first two properties are immediate from the definition and the third property is proved in \cite{Bubenik:2015}.
\end{lem}

\subsection{Multiparameter Persistence Landscapes}

Let us define the multiparameter persistence landscape in analogy with the single parameter case. The rank invariant, rank function and rescaled rank function defined above generalise naturally to multiparameter persistence modules:

\begin{defn}(Rank Invariant)

Let $M$ be an $\mathbb{R}^n$ persistence module, then  for $\vec{a}\leq \vec{b}$ the function $\beta^{\cdot,\cdot}$ giving the corresponding Betti number is the rank invariant of $M$:
$$ \beta^{\vec{a},\vec{b}} = \dim(\Ima(M(\vec{a}\leq \vec{b})))$$
\end{defn}

\begin{defn}(Multiparameter Rank Function)

The rank function $\rk : \mathbb{R}^{2n} \rightarrow \mathbb{R}$ is given by
$$\rk(\vec{b},\vec{d})= \begin{cases} \beta^{\vec{b},\vec{d}} & \text{if } \vec{b} \leq \vec{d}\\ 0 & \text{otherwise} \end{cases} $$
\end{defn}

\begin{defn}(Rescaled Multiparameter Rank Function)

The rescaled rank function $r : \mathbb{R}^{2n} \rightarrow \mathbb{R}$
$$r(\vec{m},\vec{h})= \begin{cases} \beta^{\vec{m}-\vec{h},\vec{m}+\vec{h}} & \text{if } \vec{h} \geq \vec{0}\\ 0 & \text{otherwise} \end{cases} $$
\end{defn}

One could perform statistical analysis directly to the rank function and rescaled rank function. Endowed with the natural vector space structure these functions are not stable with respect to the interleaving distance.

\begin{ex}(Rank Function Instability)
  For any $\varepsilon >  0 $ and $N\in\mathbb{N} $ the multiparameter persistence module $M = \langle (a_i,\vec{0}) \textrm{ for } i = 1,...,N \ |\ \vec{x}^\vec{\varepsilon}\cdot a_i  \textrm{ for } i = 1,...,N \rangle$ is such that $\| \rk_M\|_\infty = N$ and $d^\omega(M,0) = \varepsilon$.
\end{ex}

We wish to define a stable invariant and so we derive a landscape function from the rank invariant. 
There are several choices to make in defining the persistence landscape for the multiparameter setting. We outline three natural generalisations from the single parameter landscape:

\begin{defn}(Cartesian Product $p$-Landscape)
This choice corresponds to taking the ordinary Persistence Landscape function in each coordinate and then applying the $p$-norm, $\lambda_{p} : \mathbb{N} \times \mathbb{R}^n \rightarrow \overline{\mathbb{R}}$.
$$ \lambda_{p}(k,\vec{x}) = \|(\sup\{ h_i \geq 0 : \beta^{\vec{x}-h_i\vec{e}_i,\vec{x}+h_i\vec{e}_i} \geq k\})_i \|_{p}$$
\end{defn}

\begin{defn}(Maximal Persistence $p$-Landscape)
This landscape searches for the line through $\vec{x}$ in the parameter space over which $k$ features persist for the longest interval about $\vec{x}$ along this line, and evaluates at half the length of this interval $\lambda_{p} : \mathbb{N} \times \mathbb{R}^n \rightarrow \overline{\mathbb{R}}$.
$$ \lambda_{p}(k,\vec{x}) = \sup\{ \|\vec{h} \|_{p} \geq 0 : \beta^{\vec{x}-\vec{h},\vec{x}+\vec{h}} \geq k\}$$
\end{defn}

\begin{defn}(Uniform Persistence $p$-Landscape)
This landscape considers the maximal length over which $k$ features persist in every (positive) direction through $\vec{x}$ in the parameter space $\lambda_{p} : \mathbb{N} \times \mathbb{R}^n \rightarrow \overline{\mathbb{R}}$.
$$ \lambda_{p}(k,\vec{x}) = \sup\{ \varepsilon \geq 0 : \beta^{\vec{x}-\vec{h},\vec{x}+\vec{h}} \geq k \text{ for all } \vec{h} \geq \vec{0} \text{ with } \| \vec{h} \|_{p} \leq \varepsilon\}$$
\end{defn}




It is worth noting that when restricted to the single parameter case all three of these definitions coincide with the single parameter Persistence Landscape \cite{Bubenik:2015}. However, the Cartesian Product Landscape and Maximal Persistence Landscape are not necessarily continuous (Figure \ref{Discontinuity}). In contrast we can show that the Uniform Persistence $p$-Landscape is 1-Lipschitz.

\begin{figure}
\begin{center}
\begin{tikzpicture}

\draw (0,0) -- (5,0);
\draw (0,0) -- (0,5);
\filldraw[fill=green, fill opacity = 0.2, draw=white] (1,1) rectangle (5,5);
\node at (0.9,2) [draw,circle,fill = black, inner sep = 0.3mm, label=left:$\vec{a}$]{};
\node at (1.1,2) [draw,circle,fill = black, inner sep = 0.3mm, label=right:$\vec{a + \varepsilon}$]{};
\draw[<->, draw=red] (1.1,2) -- (1.1,1) node[midway,right]{$h_2$};
\end{tikzpicture}
\hspace{1cm}
\begin{tikzpicture}

\draw (0,0) -- (5,0);
\draw (0,0) -- (0,5);
\filldraw[fill=green, fill opacity = 0.2,draw opacity = 0 ] (1,2) rectangle (5,5);
\filldraw[fill=green, fill opacity = 0.2,draw opacity = 0] (2,1) rectangle (5,5);

\filldraw[fill = red, fill opacity = 0.3, draw opacity = 0](3,2.5) -- (2.5,2.5) arc (180:270:0.5cm) -- cycle ;
\filldraw[fill = red, fill opacity = 0.3, draw opacity = 0](3,2.5) -- (3.5,2.5) arc (0:90:0.5cm) -- cycle ;
\node at (3,2.5) [draw,circle,fill = black, inner sep = 0.3mm, label=right:$\vec{a}$]{};

\filldraw[fill = red, fill opacity = 0.3, draw opacity = 0](2.5,4.5) -- (1,4.5) -- (2.5,3) -- cycle ;
\filldraw[fill = red, fill opacity = 0.3, draw opacity = 0](2.5,4.5) -- (4,4.5) -- (3.5,5) -- (2.5,5) -- cycle ;
\node at (2.5,4.5) [draw,circle,fill = black, inner sep = 0.3mm, label=right:$\vec{b}$]{};

\filldraw[fill = red, fill opacity = 0.3, draw opacity = 0](4,1.75) rectangle (3.25,1);
\filldraw[fill = red, fill opacity = 0.3, draw opacity = 0](4,1.75) rectangle (4.75,2.5);
\node at (4,1.75) [draw,circle,fill = black, inner sep = 0.3mm, label=right:$\vec{c}$]{};
\end{tikzpicture}
\end{center}

\caption{Both diagrams are cartoons of 2-parameter free persistence modules, where the intensity of shading at each point is proportional to the dimension of the persistence module at that choice of parameter values.
The first diagram represents a module with one free generator. This simple module exhibits the discontinuity of the Cartesian Product and Maximal Persistence Landscapes.
The second diagram illustrates the \textit{half-balls} over which persistence of $k$ homological features is measured by the Uniform Persistence Landscape for different choices of $k$ and $p$-norm. At $\vec{a}$ we sketch the maximal half-ball for the 2-norm over which 2 features persist. At $\vec{b}$ we sketch the maximal half-ball for the 1-norm over which 1 feature persists. At $\vec{c}$ we sketch the maximal half-ball for the $\infty$-norm over which 1 feature persists. }
\label{Discontinuity}
\end{figure}
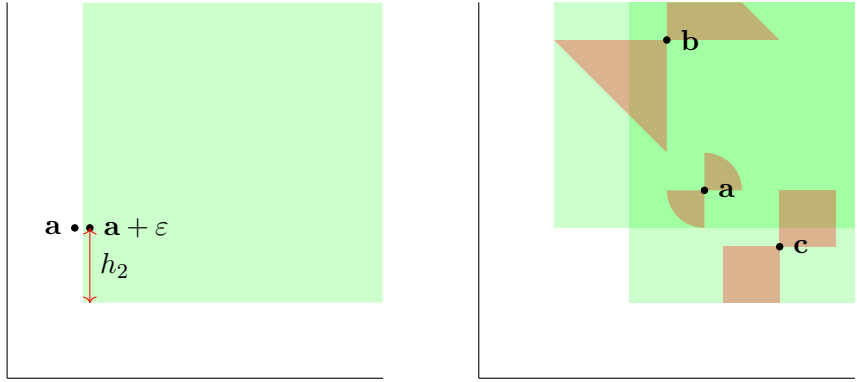

\begin{prop}

The Uniform Persistence $p$-Landscape is 1-Lipschitz for all $p \in [1,\infty]$

\end{prop}

\begin{proof}
Let $\vec{x},\vec{y} \in \mathbb{R}^n$ and let $k \in \mathbb{N}$. Without loss of generality assume that $\lambda_p(k,\vec{x}) \geq \lambda_p(k,\vec{y}) $ and that also $r = \lambda_p(k,\vec{x}) \geq \|\vec{x} - \vec{y} \|_p = \delta $. We seek to show that $\lambda_p(k,\vec{y}) \geq \lambda_p(k,\vec{x}) - \|\vec{x} - \vec{y} \|_p$.

For any $\vec{\varepsilon} \geq \vec{0}$ such that $\|\vec{\varepsilon}\|_p \leq r - \delta $ let us define $\vec{h} = (|x_i - y_i| + \varepsilon_i)_i$. Then we observe that $\|\vec{h}\|_p \leq \|\vec{x} - \vec{y}\|_p + \|\vec{\varepsilon}\|_p \leq r$ and in particular that: 
$$ \vec{x} - \vec{h} \leq \vec{y} - \vec{\varepsilon} \leq \vec{y} + \vec{\varepsilon} \leq \vec{x} + \vec{h}$$
Thus $\|\vec{h}\|_p \leq r$ means the map $M(\vec{x} - \vec{h} \leq \vec{x} + \vec{h})$ which factors through
$M( \vec{y} - \vec{\varepsilon} \leq \vec{y} + \vec{\varepsilon})$ has rank at least $k$. Since $\vec{\varepsilon}$ was arbitrary we see that $\lambda_p(k,\vec{y}) \geq r - \delta$
\end{proof}

It is clear that the discriminating power of the proposed multiparameter landscapes will be restricted by the discriminating power of the rank invariant. It is well known, that non-isomorphic modules may have the same rank invariant see Figure \ref{Non-Isomorphic}.

\begin{figure}
\begin{center}
\begin{tikzpicture}

\draw (0,0) -- (5,0);
\draw (0,0) -- (0,5);
\filldraw[fill=green, fill opacity = 0.2,draw opacity = 0 ] (1,2) rectangle (5,5);
\filldraw[fill=green, fill opacity = 0.2,draw opacity = 0] (2,1) rectangle (5,5);
\node at (2,1) [draw,circle,fill = black, inner sep = 0.3mm, label=right:$\vec{a}$]{};
\node at (1,2) [draw,circle,fill = black, inner sep = 0.3mm, label=right:$\vec{b}$]{};
\end{tikzpicture}
\hspace{1cm}
\begin{tikzpicture}

\draw (0,0) -- (5,0);
\draw (0,0) -- (0,5);
\filldraw[fill=green, fill opacity = 0.2,draw opacity = 0 ] (1,2) rectangle (5,5);
\filldraw[fill=green, fill opacity = 0.2,draw opacity = 0] (2,1) rectangle (5,5);
\node at (2,1) [draw,circle,fill = black, inner sep = 0.3mm, label=right:$\vec{a}$]{};
\node at (1,2) [draw,circle,fill = black, inner sep = 0.3mm, label=right:$\vec{b}$]{};
\node at (2,2) [draw,circle,fill = black, inner sep = 0.3mm, label=right:$\vec{c}$]{};
\end{tikzpicture}
\end{center}
\caption{The diagrams are cartoons of the 2 parameter persistence modules $M,N$, thought of as elements of $\mathbb{R}[x_1,x_2]-\textbf{Mod}$ with presentations given by: $ M = \langle(a,(2,1)) (b,(1,2))| \rangle$ and $ N = \langle(a,(2,1)) (b,(1,2)) (c,(2,2))| x_2\cdot {a} = x_1 \cdot {b} \rangle$. The rank invariant of these non-isomorphic modules coincide.}
\label{Non-Isomorphic}
\end{figure}
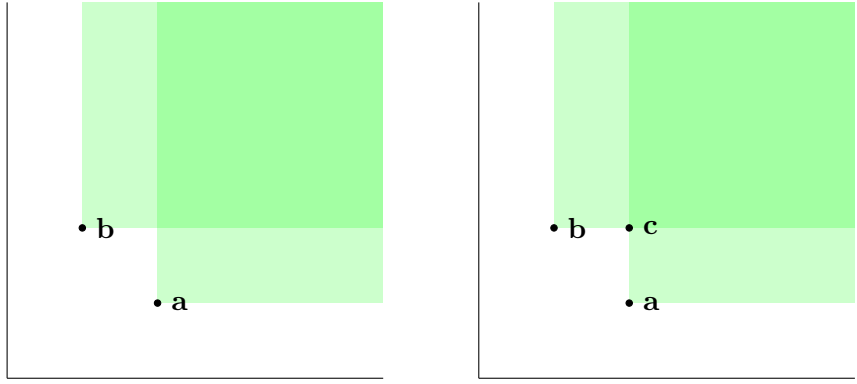

We shall proceed with the Uniform Persistence $p$-Landscape as our chosen generalisation of the single parameter persistence landscape since these landscapes are stable with respect to interleaving and the Cartesian and Maximal Landscapes are not. The discontinuity and consequent lack of stability of the Cartesian and Maximal Landscapes is related to the results in  \cite{Landi2014} regarding the stability of the fibered barcode space.
\begin{Remark}
By considering the inclusion of $n$-dimensional balls with respect to the $p$-norm and $p'$-norm, we observe that our uniform landscape functions for different choices of norm are related as follows:
$$ c_n\lambda_{p}(k,\vec{x}) \leq \lambda_{p'}(k,\vec{x}) \leq \frac{1}{C_n}\lambda_{p}(k,\vec{x})$$
Where $c_n,C_n$ are the largest constants satisfying the inclusions of the zero centred balls $B_{p'}(c_n) \subseteq B_{p}(1)$ and $B_{p}(C_n) \subseteq B_{p'}(1)$
\label{remark:qrelated}
\end{Remark}

For purposes of computation it is most convenient to use the $\infty$-landscape. In order to calculate the value of the $\infty$-landscape $\lambda_\infty(k,\vec{x})$ at the point $\vec{x}=\vec{x}_0$ one only needs to calculate $\sup\{\varepsilon\geq 0 : \beta^{\vec{x}_0-\varepsilon\vec{1},\vec{x}_0 + \varepsilon\vec{1}} \geq k\}$. Thus we only need to compute a single barcode in the fibered barcode space to compute the landscape value at a point. 

\begin{prop}
Let $M$ be a persistence module and let $\vec{L}$ be a line of slope $\vec{1}$ through the parameter space $\vec{R}^n$. Let $\iota_L : (\vec{R},\|\cdot\|_\infty) \to (\vec{R}^n,\|\cdot\|_\infty) $ denote the isometric embedding with $\iota_L(\vec{R}) = \vec{L}$ and $\iota_L(0) \in \{x_n =0\}$. Then the restriction of the uniform $\infty$-landscape of $M$ to $L$ and the single parameter persistence landscape of the persistence module $M^L$ coincide, $\lambda_{M^L}(k,t) = \lambda_{M}(k, \iota_L(t))$.
\end{prop}

\begin{proof}
Following the landscape definitions we see that if $\lambda_{M}(k, \iota_L(t)) > h$ then we have that $\beta_M^{\iota_L(t)-h\vec{1},\iota_L(t)+h\vec{1}} = \beta_{M^L}^{t-h,t+h} \geq k$ and so $\lambda_{M^L}(k,t) > h$ thus $\lambda_{M^L}(k,t) \geq \lambda_{M}(k, \iota_L(t))$. Conversely, if $\lambda_{M^L}(k,t) > h$ then $\beta_M^{\iota_L(t)-h\vec{1},\iota_L(t)+h\vec{1}} \geq k$, and since we are working with the infinity-norm if $\vec{h} \geq \vec{0}$ with $\|\vec{h}\|_\infty \leq h$ then  $\vec{h} \leq  h\vec{1}$ and so $\beta_M^{\iota_L(t)-\vec{h},\iota_L(t)+h\vec{h}} \geq k$ and thus $\lambda_{M^L}(k,t) \leq \lambda_{M}(k, \iota_L(t))$.
\end{proof}

\begin{prop}
The uniform persistence $\infty$-landscape of an interval decomposable module $M \cong \bigoplus_j \mathds{1}^{I_j}$ can be expressed as the pointwise maximum $\lambda_\infty(M)(k,\vec{x}) = \text{kmax}_j \ \lambda_\infty(I_j)(1,\vec{x})$. Note that this is only true for the uniform persistence $\infty$-landscape, and not for other choices of $p$.
 \end{prop}

The following propositions demonstrate the information one loses in deriving the Uniform Persistence Landscape from the rank invariant. In the single parameter case the landscape captures almost all the information available from the rank invariant, whilst for the multiparameter landscape the rank invariant is only recoverable from the landscapes on pairs of points spanning hypercubes.

\begin{prop}
Let $M\in \vec{vect}^{\vec{R}}$ be a pointwise-finite-dimensional, finitely generated persistence module and let $\lambda : \mathbb{N} \times \mathbb{R} \to \mathbb{R}$ be the associated persistence landscape. Then $\lambda$ determines the rank invariant almost everywhere.
\end{prop}

\begin{proof}
Let $C = \text{supp}(\xi_0 )\cup \text{supp}(\xi_1)$, then for $a,b\in (\mathbb{R}\setminus C)^2$ we recover the rank invariant by observing $\beta^{a,b} = \max \{k\in \mathbb{N} \ | \ \lambda(k, \frac{a+b}{2}) \geq \frac{b-a}{2} \}$. Since $\beta^{a,b} \geq k \implies \lambda(k, \frac{a+b}{2}) \geq \frac{b-a}{2}$ and conversely, since $a,b \notin C$ and $M$ finitely generated, we note that $\lambda(k, \frac{a+b}{2}) \geq \frac{b-a}{2} \implies \lambda(k, \frac{a+b}{2}) > \frac{b-a}{2}$ and thus $\beta^{a.b}\geq k$.
\end{proof}

\begin{prop}
\label{Hypercube}
Let $M\in \vec{vect}^{\vec{R}^n}$ be a pointwise-finite-dimensional, finitely presented 
persistence module and let $\lambda_\infty : \mathbb{N} \times \mathbb{R}^n \to \mathbb{R}$ be the associated uniform persistence landscape. Let $C \subset \mathbb{R}^n$ be the set of points sharing a coordinate with $\text{supp}(\xi_0) \cup \text{supp}(\xi_1)$:
$$ C = \{\vec{x} \in \mathbb{R}^n \ |\ x_j = a_j \text{ for some } \vec{a} \in \text{supp}(\xi_0) \cup \text{supp}(\xi_1) \} $$
Then using $\lambda_\infty$ we can recover the rank invariant of $M$ on the set:
$$L = \{(\vec{a},\vec{b}) \in (\mathbb{R}^n\setminus C)^2 \  | \ \vec{a}\leq \vec{b}, \ [\vec{a},\vec{b}] \text{ is a hypercube }  \}$$
\end{prop}

\begin{proof}
The rank invariant on the set $L$ is derived from the landscape as follows: 
$$\beta^{\vec{a},\vec{b}} = \max \left\{ k \in \mathbb{N} \ \bigg| \ \lambda_\infty(k,\frac{\vec{a}+\vec{b}}{2}) \geq \left\| \frac{\vec{b}-\vec{a}}{2}\right\|_\infty \right\}$$

For if $\beta^{\vec{a},\vec{b}} \geq k$ then all morphisms starting and ending in the hypercube $[\vec{a},\vec{b}]$ have rank at least $k$. Thus for the centre of the hypercube $\vec{c} = \frac{\vec{a}+\vec{b}}{2}$ we see $\lambda_\infty(k,\vec{c}) \geq \|\frac{\vec{b}-\vec{a}}{2}\|_\infty$. Conversely if $\lambda_\infty(k,\vec{c}) \geq \|\frac{\vec{b}-\vec{a}}{2}\|_\infty$, since $\vec{a},\vec{b} \notin C$ and $M$ finitely presented, then in fact $\lambda_\infty(k,\vec{c}) > \|\frac{\vec{b}-\vec{a}}{2}\|_\infty$, and so $\beta^{\vec{a},\vec{b}} \geq k$.

\end{proof}

The following example illustrates two modules with distinct rank invariants which are not distinguished by their uniform persistence landscapes. 

\begin{ex}
  Denote by $[\vec{a},\vec{b}]$ the rectangular module with opposite vertices $\vec{a},\vec{b}$. Let $M$, $N$ be $2$-parameter persistence modules and rectangular decomposable, with $M =[(0,1), (10,2)] \oplus [(4,1), (6,2)]$ and let $N =[(0,1), (6,2)] \oplus [(4,1), (10,2)]$. These two modules have different rank invariant but the same Uniform Persistence Landscapes illustrated in Figure \ref{LandscapeVsRankInvariant}.
\end{ex}

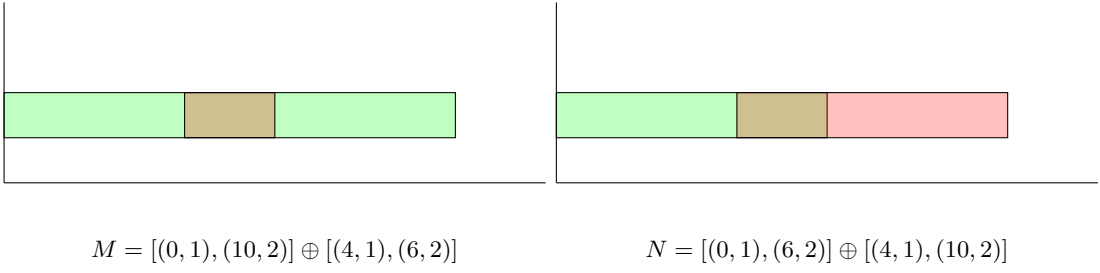
\begin{figure}
    \centering
    \begin{tikzpicture}[scale = 1.2]
    
    \draw (0,0) -- (6,0);
    \draw (0,0) -- (0,2);
    
    \filldraw[fill=green, fill opacity = 0.25] (0,.5) rectangle (5,1);
    \filldraw[fill=red, fill opacity = 0.25] (2,.5) rectangle (3,1);
    \node [anchor=south] (I) at (3,-1) {\footnotesize{$M =[(0,1), (10,2)] \oplus [(4,1), (6,2)]$}};
    \end{tikzpicture}
    \begin{tikzpicture}[scale = 1.2]

    \draw (0,0) -- (6,0);
    \draw (0,0) -- (0,2);
    
    \filldraw[fill=green, fill opacity = 0.25] (0,.5) rectangle (3,1);
    \filldraw[fill=red, fill opacity = 0.25] (2,.5) rectangle (5,1);
    \node [anchor=south] (I) at (3,-1) {\footnotesize{$N =[(0,1), (6,2)] \oplus [(4,1), (10,2)]$}};
    \end{tikzpicture}
    
    \caption{We illustrate modules that have different rank invariant but the same Uniform Persistence Landscapes. (The first summand of each module is shaded green and the second summand red.)}
\label{LandscapeVsRankInvariant}
\end{figure}

One would hope that a multiparameter module invariant could distinguish the modules in the previous example. The uniform persistence landscape fails to distinguish these modules since the rank invariants of these modules coincide on all pairs of points spanning hypercubes. This in turn occurs since the overlap between the summands in the $x_1$-coordinate is greater than their significance in the $x_2$-coordinate. Thus altering $x_1,x_2$ at the same rate we cannot detect the interaction between the summands with the uniform landscape. A simple reparametrisation scaling parameters $x_i$ appropriately would allow us to distinguish these modules and motivates the following definition.

\begin{defn}($\vec{w}$-Weighted Persistence Landscape)
Let $\vec{w}\in \{\vec{u} \in \mathbb{R}^n | u_i>0 , \|\vec{u}\|_\infty = 1 \}$ be a weighting vector corresponding to a rescaling of the parameter space $\mathbb{R}^n$. Define the $\vec{w}$-weighted infinity norm to be $\| \vec{h} \|^\vec{w}_{\infty} = \| (w_ih_i)_i \|_{\infty}$. The $\vec{w}$-Weighted Persistence Landscape is a function $\lambda_{\vec{w}} : \mathbb{N} \times \mathbb{R}^n \rightarrow \overline{\mathbb{R}}$.
$$ \lambda_{\vec{w}}(k,\vec{x}) = \sup\{ \varepsilon \geq 0 : \beta^{\vec{x}-\vec{h},\vec{x}+\vec{h}} \geq k \text{ for all } \vec{h} \geq \vec{0} \text{ with } \| \vec{h} \|^\vec{w}_{\infty} \leq \varepsilon\}$$
\end{defn}

\begin{Remark}
The Uniform Persistence $\infty$-Landscape is the $\vec{1}$-Weighted Persistence Landscape.
\end{Remark}

\begin{defn}($\vec{w}$-Rescaling)
Let $\varphi_\vec{w} \in \text{Aut}(\mathbb{R}^n)$ denote the invertible rescaling $\varphi_\vec{w}(\vec{x})=(w_ix_i)_i$ for $\vec{w}\in \{\vec{u} \in \mathbb{R}^n | u_i>0 , \|\vec{u}\|_\infty = 1 \}$.
\end{defn}

The following proposition makes precise the relationship between the weighted landscape and rescaling the parameter space.

\begin{prop}
\label{Rescaling}
Let $\vec{w}$ be a rescaling vector and let $\lambda_{\vec{w}}$ denote the function taking a module to its $\vec{w}$-Weighted Persistence Landscape. Let $(\varphi_\vec{w})^\ast$ denote the pull back of $\varphi_\vec{w}$.  Then the following diagram commutes:\\
\[
\begin{tikzcd}
\vec{vect}^{\vec{R}^n} \arrow{rr}{(\varphi_\vec{w})^\ast} \arrow{d}{\lambda_\vec{1}} & & \vec{vect}^{\vec{R}^n} \arrow{d}{\lambda_\vec{w}}\\
L_p(\mathbb{N}\times \mathbb{R}^n)  \arrow{rr}{(\id \times \varphi_\vec{w})^\ast} & & L_p(\mathbb{N}\times \mathbb{R}^n)
\end{tikzcd}
\]
\end{prop}

\begin{proof}
Let $M \in \vec{vect}^{\vec{R}^n}$, it is a straight forward definition chase to see the diagram commutes. For convenience of notation let us use shorthand notation for component wise multiplication $\vec{a}\odot\vec{b} = (a_ib_i)_i$.
\begin{align*}
    \lambda_{\vec{w}}(M \circ \varphi_{\vec{w}})(k,\vec{x}) &= \sup \{\varepsilon \geq 0 \ |\  \beta^{\vec{x}-\vec{h},\vec{x}+\vec{h}}_{M \circ \varphi_{\vec{w}}} \geq k, \text{ for all } \vec{h} \geq \vec{0} \text{ with } \| \vec{h} \|^{\vec{w}}_{\infty} \leq \varepsilon \} \\
    &= \sup \{\varepsilon \geq 0 \ |\  \beta^{\vec{w}\odot(\vec{x}-\vec{h}),\vec{w}\odot(\vec{x}+\vec{h})}_{M} \geq k, \text{ for all } \vec{h} \geq \vec{0} \text{ with } \| \vec{w}\odot\vec{h} \|_{\infty} \leq \varepsilon \} \\
    &= \sup \{\varepsilon \geq 0 \ |\  \beta^{\vec{w}\odot\vec{x}-\vec{t},\vec{w}\odot\vec{x}+\vec{t}}_{M} \geq k, \text{ for all } \vec{t} \geq \vec{0} \text{ with } \| \vec{t} \|_{\infty} \leq \varepsilon \} \\
    &= \lambda_{\vec{1}}(M)(k,\vec{w}\odot\vec{x}) = ((\textrm{id} \times \varphi_\vec{w})^\ast \circ \lambda_{\vec{1}}(M))(k,\vec{x})
\end{align*}
\end{proof}

The example illustrated in Figure \ref{LandscapeVsRankInvariant} exhibits the dependence on the relative scaling of the significant parameter values in a multiparameter persistence module and highlights the importance of normalisation or choice of weighting vectors $\vec{w}$ in practical applications.

\section{Stability and Injectivity}
\label{Stability}

In this section we shall show that the multiparameter landscapes are stable with respect to the interleaving distance and persistence weighted Wasserstein distance. We will then provide an injectivity result that shows the collection of weighted persistence landscapes derived from a persistence module contains almost all the information in the rank invariant of that module.

Let us define a $q$-distance on the space of multiparameter landscapes completely analogously with the definition as in \cite{Bubenik:2015} where we implicitly are viewing our landscapes as elements of $L_q(\mathbb{N}\times \mathbb{R}^n)$. Our landscapes are all measurable since they are continuous, however they may be unbounded. We can either choose to permit infinite distances or alternatively truncate our landscapes to a bounded region in order to that all our distances are finite.
\begin{defn}($q$-Landscape Distance)

Let $M, M'$ be persistence modules then define the $q$-landscape distance to be:
$$ d_{\lambda_p}^{(q)}(M,M') = \| \lambda_{p}(M) - \lambda_{p}(M') \|_{q} $$
\end{defn}

\subsection{Stability}




We now show that unlike the rank function and rank invariant, the infinity norm of multiparameter persistence landscapes is stable with respect to the interleaving distance. 

\begin{thm}(Multiparameter Uniform $\infty$-Landscape Stability)
\label{thm:stability}
Let $M , M' \in \textbf{vect}^{\textbf{R}^n}$ be multiparameter persistence modules and $\omega$ the sublinear projection given by $\omega_{\Gamma} = \| \Gamma - I \|_{\infty}$, then the $\infty$-Landscape distance of the Uniform Persistence $\infty$-Landscapes is bounded by the induced interleaving distance.
$$  d_{\lambda_\infty}^{(\infty)}(M,M') \leq d^{\omega}(M,M') $$

\end{thm}

\begin{proof}

Suppose $M,M'$ are $\varepsilon$-interleaved with respect to $\omega$, and let $(\Gamma, K)$ realise an $\varepsilon$-interleaving. Let $\vec{x} \in \mathbb{R}^n$ and assume without loss of generality that $ r = \lambda_{\infty}(k,\vec{x}) \geq \lambda_{\infty}'(k,\vec{x})$ and that also $ \lambda_{\infty}(k,\vec{x}) \geq \varepsilon $.

For any $\vec{h} \in B^{\geq}_\infty(\vec{0},r-\varepsilon) $ we have that $\vec{h} + \varepsilon \vec{1} \in B^{\geq}_\infty(\vec{0},r)$. Since $r = \lambda_{\infty}(k,\vec{x})$ we know that the map $M(\vec{x} - (\vec{h} + \varepsilon \vec{1}) \leq \vec{x} + (\vec{h} + \varepsilon \vec{1}))$ has rank at least $k$.

The $(\Gamma, K)$ $\varepsilon$-interleaving gives rise to commutative diagram:

\vspace{5mm}

\begin{tikzcd}[cramped, column sep=tiny, font=\footnotesize]
M(\vec{x} - (\vec{h} + \varepsilon \vec{1}))
\arrow[ddr]
\arrow[rrrr] & & & & M(K(\vec{x} + \vec{h})) \arrow[r] & M(\vec{x} + (\vec{h} + \varepsilon \vec{1})) \\
&&&&&\\
&  M'(\Gamma(\vec{x} - (\vec{h} + \varepsilon \vec{1}))) \arrow[r] & M'(\vec{x} - \vec{h}) \arrow[r] & M'(\vec{x} + \vec{h}) \arrow[uur] & &
\end{tikzcd}

\vspace{5mm}

Thus we see that the map $M'(\vec{x} - \vec{h} \leq \vec{x} + \vec{h})$ has rank at least $k$.
\end{proof}

\begin{cor}(Multiparameter Sublevel Set $\infty$-Landscape Stability Theorem)
Let $ f,g : X \rightarrow \mathbb{R}^n $ then the sublevel set persistence modules satisfy:
$$ d_{\lambda_\infty}^{(\infty)}(M(f),M(g)) \leq \|f-g\|_{\infty} $$
\end{cor}

\begin{proof}
$M(f),M(g)$ are $\| f-g\|_{\infty}$ interleaved with respect to $\omega_{\Gamma} = \|\Gamma - I \|_{\infty}$
\end{proof}

We observed in Remark \ref{remark:qrelated} that the multiparameter $p$-landscapes are related to the $\infty$-landscapes by constant factors. Thus the $\infty$-landscape distance between a pair of $p$-landscapes is bounded by some constant multiple of the interleaving distance. Note also that in the situation where we truncate our landscapes to a bounded region $R \subset \mathbb{R}^n$, the $\infty$-landscape distance stability result yields a coarse bound for the $p$-landscape distance: the product of the measure of the bounded region $|R|$ and the $\infty$-landscape distance. 
Hence working over a bounded region $R$ the $q$-landscape distance between two $p$-landscapes is bounded by a constant factor of the interleaving distance for any choice of $p$ and $q$, with the constant dependent on $p,q, |R|$.

The weighted landscapes also satisfy stability with respect to the interleaving distance. This can be shown directly or using Proposition \ref{Rescaling} and the stability result in the uniform case.

\begin{cor}(Multiparameter $\vec{w}$-Weighted $\infty$-Landscape Stability)
Let $M,N$ be multiparameter modules and let $\omega, \omega^{\vec{w}}$ be the sublinear projections given by $\omega_{\Gamma} = \| \Gamma - \id \|_{\infty},\ \omega^{\vec{w}}_{\Gamma} = \|\Gamma - \id \|_{\infty}^{\vec{w}}$ respectively. Then the induced interleaving distances bound the $\infty$-landscape distance:

$$ \| \lambda_\vec{w}(M) - \lambda_\vec{w}(N)\|_\infty  \leq d^{\omega^\vec{w}}(M,N) \leq d^{\omega}(M,N)$$

\end{cor}


The $q$-landscape distance restricted to interval decomposable modules is stable with respect to the persistence weighted $q$-Wasserstein distance.

\begin{prop}($q$-Landscape Distance Stability of Interval Decomposable Modules)
Let $M,N$ be interval decomposable multiparameter modules with finite barcodes, and recall $d_{\overline{W}_q}$ the persistence weighted $q$-Wasserstein distance. The $q$-landscape distance is stable with respect to the persistence weighted $q$-Wasserstein distance:
$$ d_{\lambda_\infty}^{(q)}(M,N) \leq d_{\overline{W}_q}(M,N)$$
\end{prop}
\begin{proof}
Let us use the shorthand notation $\lambda_M = \lambda_\infty(M)$, and suppose $M,N$ have barcodes $\{I_j\ \ | \ j \in \mathcal{J}\}$ and $\{J_\kappa \ \ | \ \kappa \in \mathcal{K}\}$ with equal cardinality (else append empty intervals). Recall that the landscape for $M$ can be expressed as a pointwise maximum, $\lambda_M(k,\vec{x}) = \textrm{kmax}_\mathcal{J} \ \lambda_{\mathds{1}^{I_j}}(1,\vec{x})$. Let $\sigma : \mathcal{J} \to \mathcal{K}$ be any bijection realising the persistence weighted $q$-Wasserstein distance.
    \begin{align*}
         d_{\lambda_\infty}^{(q)}(M,N)^q = \|\lambda_M -\lambda_N \|_q^q &= \sum_{k=1}^\infty \int_{\mathbb{R}^n} |\lambda_M(k,\vec{x}) -\lambda_N(k,\vec{x}) |^q d\mu \\
         &=  \int_{\mathbb{R}^n} \sum_{k=1}^\infty |\textrm{kmax}_\mathcal{J} \ \lambda_{\mathds{1}^{I_j}}(1,\vec{x}) -\textrm{kmax}_\mathcal{K} \ \lambda_{\mathds{1}^{J_\kappa}}(1,\vec{x}) |^q d\mu \\
         &\leq \int_{\mathbb{R}^n} \sum_{j\in \mathcal{J}} | \lambda_{\mathds{1}^{I_j}}(1,\vec{x}) - \lambda_{\mathds{1}^{J_{\sigma(j)}}}(1,\vec{x}) |^q d\mu \\
         &\leq \int_{\mathbb{R}^n} \sum_{j\in \mathcal{J}} \varepsilon_j^q \mathds{1}_{\{I_j \cup J_{\sigma(j)}\}} d\mu \\
         &= \sum_{j\in \mathcal{J}} |I_j \cup J_{\sigma(j)}|\varepsilon_j^q = d_{\overline{W}_q}(M,N)^q
    \end{align*}
    
The inequality between the second and third line follows from the general fact that for any $\vec{u},\vec{v} \in \mathbb{R}^n$ the sum $\sum |u_i - v_i|^q$ is minimised by ordering the components of each tuple. The fourth line bounds the third line by Theorem \ref{thm:stability} applied to the matched interval summands.
\end{proof}

\subsection{Injectivity}

We now show that the collection of weighted landscapes associated to a module preserves almost all the information contained in the rank invariant. 

\begin{prop}
Let $\vec{vect}_{\text{fin}}^{\vec{R}^n}$ denote the collection of pointwise finite dimensional, finitely presented persistence modules. Let us define an equivalence relation on $\vec{vect}_{\text{fin}}^{\vec{R}^n}$ identifying $M \sim N$ if the rank invariant of $M$ and $N$ coincide almost everywhere. Then the map $\lambda : M \mapsto \{(\vec{w}, \lambda_\vec{w}(M))\}$ is well defined and injective on the quotient space $\vec{vect}_{\text{fin}}^{\vec{R}^n}/\sim$. Moreover, equipping the quotient space with the distance induced by the interleaving distance, and the weighted landscape space with the metric: $$d(\lambda(M),\lambda(N)) = \sup_\vec{w}\{ \| \lambda_\vec{w}(M) - \lambda_\vec{w}(N)\|_\infty\}$$ then this map is $1$-Lipschitz.
\end{prop}

\begin{proof} $\newline$
    \begin{enumerate}

    \item  The map $\lambda$ is well defined:
    
    Suppose $M,N$ are modules such that $\lambda_\vec{1}(M) \neq \lambda_\vec{1}(N)$ we will show that the rank invariant of these modules differ on a set of positive measure. Without loss of generality let $r = \lambda_\vec{1}(M)(k,\vec{x}) > \lambda_\vec{1}(N)(k,\vec{x})$ and let $\varepsilon < \lambda_\vec{1}(M)(k,\vec{x}) - \lambda_\vec{1}(N)(k,\vec{x})$. Since $\lambda_\vec{1}(N)(k,\vec{x})<r$ there is some $\vec{h}$ with $\| \vec{h}\|_\infty<r$ such that $\beta_N^{\vec{x}-\vec{h},\vec{x}+\vec{h}} < k$.
    Consider any element $(\vec{a},\vec{b})$ of the open set $B^{<}(\vec{x}-\vec{h},\varepsilon)\times B^{>}(\vec{x}+\vec{h},\varepsilon) \subset \mathbb{R}^n \times  \mathbb{R}^n$. Then $r = \lambda_\vec{1}(M)(k,\vec{x}) \implies \beta_M^{\vec{a},\vec{b}} \geq k$ and $\vec{a} \leq \vec{x}-\vec{h}\leq \vec{x}+\vec{h}\leq \vec{b} \implies \beta_N^{\vec{a},\vec{b}} \leq \beta_N^{\vec{x}-\vec{h},\vec{x}+\vec{h}} < k$. Thus the rank invariants differ on a set of positive measure.
    
    \item  The map $\lambda$ is injective:
    
    Let $M \in \vec{vect}_{\text{fin}}^{\vec{R}^n}$ and recall Proposition \ref{Hypercube}. For all $\vec{a} < \vec{b}$ there is some rescaling vector $\vec{w}$ such that $[\varphi_\vec{w}(\vec{a}),\varphi_\vec{w}(\vec{b})]$ spans a hypercube. Thus we can recover the rank invariant of $M$ almost everywhere from the collection $\{(\vec{w}, \lambda_\vec{w}(M))\}$.
    
    \item  The map $\lambda$ is $1$-Lipschitz:
    
    This is an immediate consequence of Multiparameter $\vec{w}$-Weighted $\infty$-Landscape Stability.
    \end{enumerate}
\end{proof}

Since $\lambda$ is $1$-Lipschitz we can compute a lower bound on the interleaving distance between modules from the collection of weighted landscapes. We would be interested to investigate further the relationship between the landscape distance and the interleaving distance to understand when the landscape distance provides a good lower bound for the interleaving distance.





\section{Statistics on Multiparameter Landscapes}
\label{Statistics}

A principal advantage of working with landscapes as a summary statistic for our data is that we are always able to take the pointwise mean of a collection of landscapes. The space of persistence landscapes endowed with the $q$-landscape distance is naturally a subspace of Lebesgue space, a Banach space. We would like to perform statistical analysis on a set landscapes produced from data sets to distinguish significant topological signals from sampling noise.  
In this section we shall review relevant results from the theory of Banach Space valued random variables, and then apply these results to multiparameter persistence landscapes. We attain the same collection of results enjoyed by the single parameter persistence landscape established in \cite{Bubenik:2015}.

\subsection{Probability in Banach Spaces}

Let us begin by defining some notation. Let $(\mathcal{B},\| \cdot \|)$ denote a real, separable Banach Space with topological dual space $\mathcal{B}^{*}$. Let $ V : (\Omega,\mathcal{F},\mathbb{P}) \to \mathcal{B}$ denote a Borel measurable random variable. The covariance structure of such a random variable is defined to be the set of expectations $$\{\mathbb{E}[(f(V)-\mathbb{E}[f(V)])(g(V)-\mathbb{E}[g(V)])] : f,g \in \mathcal{B}^{*}\}$$

In order to take expected values of Banach valued random variables we require the notion of the Pettis Integral, which is an extension of the Lebesgue integral to functions on measure spaces taking values in normed spaces. We shall briefly introduce the properties of this integral and existence criteria, the essence of which is built upon reducing the problem to integrability of $\mathbb{R}$-valued functions.

\begin{defn}(Scalarly Integrable)\cite{Geitz1981}

A function $V: (\Omega, \mathcal{F}, \mu) \to \mathcal{B}$ is scalarly integrable if for all $f \in \mathcal{B}^{*}$ we have that $f(V)\in L_{1}(\mu)$ 

\end{defn}

\begin{defn}(Pettis Integrable)\cite{Geitz1981}

A scalarly integrable function $V: (\Omega, \mathcal{F}, \mu) \to \mathcal{B}$ is Pettis integrable if for all $E \in \mathcal{F}$ there is an element $I_V(E)\in \mathcal{B}$ such that: 
$$ \int_{E} f(V) d\mu = f(I_V(E)) \text{ for all }f\in \mathcal{B}^{*}$$

The set function $ I_V : \mathcal{F} \to \mathcal{B}$ is called the Pettis Integral of $V$ with respect to $\mu$. We may also refer to $I_V(\Omega)$ as the Pettis Integral of $V$ and denote this by $I_V$.

\end{defn}



\begin{thm} \cite{Musia2015}(Theorem 5.4)

If $\mathcal{B}$ does not contain an isomorphic copy of $(c_0,\| \cdot \|_\infty)$ then each strongly measurable and scalarly integrable $\mathcal{B}$-valued function is Pettis Integrable.

\end{thm}

Note that for a separable Banach Space the notions of weak and strong measurability coincide. Thus the previous theorem gives a sufficient criterion for Pettis Integrability in the setting of multiparameter persistence landscapes endowed with the $q$-norm for $q\in [1,\infty)$ for which the underlying Banach space is separable.

\begin{cor}\cite{Bubenik:2015}

Let $V: (\Omega, \mathcal{F}, \mu) \to \mathcal{B}$ with $\mathcal{B}$ real and separable. If $\mathbb{E}^{\mu}[\|V\|] < \infty $ then $V$ has a Pettis Integral and $\|I_V(\Omega)\| \leq \mathbb{E}^{\mu}[\|V\|] $

\end{cor}

\begin{thm}(Strong Law of Large Numbers)\cite{Talagrand2011}(Corollary 7.10)

Let $V_i$ be i.i.d copies of $V: (\Omega, \mathcal{F}, \mathbb{P}) \to \mathcal{B}$ and let $S_n = \sum_{i=1}^n V_i $. Then $\mathbb{E}[\|V\|] < \infty$ if and only if:
 $$ \frac{S_n}{n} \to I_V(\Omega) \text{ almost surely as } n \to \infty $$  

\end{thm}

\begin{defn}

We say a $\mathcal{B}$-valued random variable $X$ is Gaussian if for all $f\in \mathcal{B}^*$ the real valued random variable $f(X)$ is Gaussian with mean zero. Note that such a Gaussian random variable is determined by its covariance structure.

\end{defn}

The next result only applies for a certain class of Banach spaces. The type and cotype of a Banach space can be thought loosely of as a measure of how close that Banach space is to a Hilbert space. For $q\in [1,2]$ the Lebesgue space $L_q$ has type $q$ and cotype $2$, and for $q\in [2,\infty)$ the Lebesgue space $L_q$ has type $2$ and cotype $q$. 

\begin{thm}(Central Limit Theorem)\cite{Hoffmann-Jørgensen1976}

Let $\mathcal{B}$ be a Banach space of type 2 and $V: (\Omega, \mathcal{F}, \mathbb{P}) \to \mathcal{B}$. If $I_V = 0$ and $\mathbb{E}[\|V\|^2] < \infty$ then $\frac{1}{\sqrt{n}}S_n$ converges weakly to a Gaussian random variable with the same covariance structure as $V$.

\end{thm}

\subsection{Convergence Results for Multiparameter Landscapes}

We shall take the same probabilistic approach as in \cite{Bubenik:2015} in viewing multiparameter landscapes derived from a data set as a Banach space valued random variable. The model for applying statistical analysis to persistence landscapes will likely trace the following general setup:

Suppose $X$ is a Borel measurable random variable on some probability space $(\Omega, \mathcal{F}, \mathbb{P})$ thought of as sampling data from some distribution. Further let $\Lambda = \Lambda(X)$ denote the multiparameter persistence landscape associated to some multifiltration of the data $X$, so that in summary $\Lambda : (\Omega, \mathcal{F}, \mathbb{P}) \to L_q(\mathbb{N}\times\mathbb{R}^n)$ is a random variable taking values in a real, separable Banach Space.

Let $\{X_i\}$ be i.i.d copies of $X$ and $\{\Lambda_i\}$ their associated landscapes. Denoting the pointwise mean of the first $n$ landscapes by $\overline{\Lambda}^n$ and applying the general theory of probability in Banach spaces presented above we attain several results. Observe that in practice we may be required to truncate our multiparameter landscapes to a bounded region in order to satisfy the finiteness criteria in the convergence results. 

\begin{thm}(Strong Law of Large Numbers)

With our notation as in the above discussion $\overline{\Lambda}^n \to I_\Lambda$ almost surely if and only if $\mathbb{E}[\|\Lambda\|] < \infty$.

\end{thm}

\begin{thm}(Central Limit Theorem)

Let us consider the landscapes endowed with the $q$-Landscape distance for $q \geq 2$. Suppose $\mathbb{E}[\|\Lambda\|] < \infty$ and $\mathbb{E}[\|\Lambda^2\|] < \infty$, then $\sqrt{n}(\overline{\Lambda}^n - I_\Lambda(\Omega)) $ converges weakly to a Gaussian random variable $G(\Lambda)$ with the same covariance structure as $\Lambda$.

\end{thm}

The central limit theorem for the landscapes induces a central limit theorem for associated real valued random variables and facilitates the computation of approximate confidence intervals.

\begin{cor}

Let us consider the landscapes endowed with the $q$-Landscape distance for $q \geq 2$. Suppose $\mathbb{E}[\|\Lambda\|] < \infty$ and $\mathbb{E}[\|\Lambda^2\|] < \infty$. Furthermore let $f\in L_q(\mathbb{N}\times\mathbb{R}^n)^* \cong L_p(\mathbb{N}\times\mathbb{R}^n)$, so that $Y = f(\Lambda)$ is a real valued random variable. Then $\sqrt{n}(\overline{Y}^n - \mathbb{E}[Y]) \to \mathcal{N}(0,\text{Var}(Y))$ converges in distribution.

\end{cor}

\begin{cor}(Approximate Confidence Intervals)

Suppose $Y$ is a real-valued random variable attained from a functional applied to the multiparameter landscape $\Lambda$ satisfying the conditions of the previous Corollary.
Let $\{Y_i\}_{i=1}^n$ be i.i.d. instances of this random variable and $S_n^2 = \frac{1}{n-1} \sum_{i=1}^n (Y_i - \overline{Y}_n)^2$ the sample variance.
An approximate $(1-\alpha)$ confidence interval for $\mathbb{E}[Y]$ is given by: $[\overline{Y}_n- z_{\frac{\alpha}{2}}\frac{S_n}{\sqrt{n}} ,\overline{Y}_n+z_{\frac{\alpha}{2}}\frac{S_n}{\sqrt{n}}]$, where  $z_{\frac{\alpha}{2}}$ is the $\frac{\alpha}{2}$ critical value for the normal distribution.
\end{cor}

In practice, a functional of choice could be given by integrating the landscapes over a subset $E$ of the parameter domain $f_E(\Lambda) = \int_E \Lambda \ d\mathbb{P}$. These functionals can be used to establish the significance of homological features in different regions of the parameter space.

We remark that recent work has attained confidence bands for single parameter persistence landscapes \cite{Chazal2014StochasticCO} \cite{Chazal2013Bootstrap}. It would be interesting to see similar analysis performed in the multiparameter setting. 

\section{Example Computations and Machine Learning Applications}
\label{Computations}

In this section we shall present example computations of multiparameter persistence landscapes and demonstrate a simple application of machine learning to the persistence landscapes. We use the RIVET software for computations of 2-parameter persistence modules presented in \cite{Lesnick2015}. RIVET supports the fast computation of multigraded Betti-numbers and an interactive visualisation for 2-parameter persistence modules. The software computes a data structure associated to a module which facilitates real time queries of the fibered barcode space. As far as we know, RIVET is the only publicly available TDA software package supporting multiparameter persistent homology calculations.

The software supports a range of input formats including: point cloud, metric space, algebraic chain complex, and explicit bifiltered complex. In particular we shall use the software to calculate and query the fibered barcode associated to a module along a selection of one dimensional slices of the parameter space. Further details of the software may be found in \cite{Lesnick2015}.


Computation of the module with RIVET is the most computationally expensive procedure in our calculations. Details of the time and space complexity of the algorithm may be found in \cite{Lesnick2015}, loosely if $m$ denotes the size of the filtered complex associated to the input data, in the worst case one requires time $O(m^5)$ and space $O(m^5)$ to compute the data structure which admits fast queries of the fibered barcode space. 

In theory, since our landscape is derived solely from the rank invariant, we need not calculate the full module and fibered barcode space. Recall that the value of the multiparameter uniform persistence $\infty$-landscape at each point can be calculate using the single parameter persistence landscape associated to the line of slope $\vec{1}$ passing through that point. Thus we could reduce the computation of the multiparameter landscape in any dimension to repeated single parameter persistent homology calculations. This reduction would be highly parallelizable and likely to provide significant speedup.

\begin{prop}
Let $M \in \vec{vect}^{\vec{R}^2}$ be a multiparameter persistence module derived from a simplicial complex with $m$ simplices. Let $\varepsilon$ be some tolerance value and $[0,R]\times[0,R] \subset \mathbb{R}^2$ a subset of our parameter space. Then we can compute an $\varepsilon$-approximate $\lambda_M^{(\varepsilon)}$ to the uniform persistence landscape  $\lambda_M$ of $M$ on the region $[0,R]\times[0,R]$ in time $O(m^3 \frac{R}{\varepsilon}).$ Our approximation is with respect to the infinity norm $\|\lambda_M^{(\varepsilon)} - \lambda_M\|_\infty \leq \varepsilon$
\end{prop}
\begin{proof}
    Divide the region $[0,R]\times[0,R]$ into a grid of spacing $\varepsilon$. It suffices to calculate the values of the landscape on this grid since the landscape functions are 1-Lipschitz and so we can extend the grid values to an $\varepsilon$-approximate function on $[0,R]\times[0,R]$.
    Thus we reduce our computation to the computation of $\frac{2R}{\varepsilon}$ single parameter landscapes corresponding to the collection of $\frac{2R}{\varepsilon}$ slope $\vec{1}$ lines passing through the points of the grid.
    
    Given birth-death pairs \cite{BubenikDlotko2017} provides an algorithm to compute the persistence landscapes in time $O(m^2)$. It is well known from \cite{books/daglib/0025666} that one can produce birth-death pairs from a filtration of size $m$ in time $O(m^3)$. Hence the result follows.
\end{proof}

It is possible that the above time estimate for the landscape computation could be improved by using vineyard style updates between the single parameter landscapes \cite{Cohen-Steiner2006}. Moreover it may be that in practical applications, computing the module with RIVET in time $O(m^5)$ and using the fibered barcode queries will be faster than the computation of a series of single parameter landscapes. Note also that the $\frac{2R}{\varepsilon}$ single parameter landscape calculations are independent and so can be computed in parallel. We postpone comparisons of different computational algorithms, benchmarking, and efficient implementation to follow up work.

One may want to utilise machine learning algorithms with landscape functions as a collection of features for a data set. Recall that if we consider the persistence landscapes associated with the 2-Landscape distance then we are naturally in the setting of a Hilbert Space. The inner product on this space is positive definite on the space of persistence landscapes. As such we may use this kernel to learn non-linear relationships in our data and then apply convex optimisation techniques to an SVM.

Another point to note is that integrating an $n$-dimensional landscape over a finite resolution gives an $n$-dimensional array as a summary of our data to which one could apply a convolutional neural network. This transform from landscape to multidimensional array will satisfy stability with respect to the landscape distance. A similar approach is used in \cite{Adams2015} to produce a \textit{persistence image} from a persistence diagram.

We provide three computational examples together with the application of a basic statistical test and standard SVM classifier. Our examples demonstrate that the multiparameter landscape is sensitive to both topology and geometry. We do not claim that the multiparameter landscape is the optimal analytic tool to perform the various tasks in our examples, rather we demonstrate a range of potential applications.


\subsection{Concentric Circles}

Our first example will look at pointclouds sampled from densities concentrated around a pair of concentric circles with radii $1$ and $3$ respectively. We colour the points from each circle in two distinct ways. Colouring A assigns the large circle colour parameter $0.5$ and the small circle colour parameter $1.5$. Colouring B assigns the small circle  colour parameter $0.5$ and the large circle colour parameter $1.5$. We examine how the multiparameter landscapes differ depending on the colouring of the circles. For each colouring we perform $30$ samples, each sample consisting of $100$ points uniformly sampled from each circle Figure \ref{subfig:CirclePointclouds}.

We produce a filtration on each pointcloud with the Rips filtration in the first parameter and the colour parameter in the second parameter. Thus at parameter value $(r,c)\in\mathbb{R}^2$, we have the space $X_{(r,c)} = \text{VR}(\mathcal{P}_{c},r)$ where $\mathcal{P}_{c}$ denotes the sampled points with colour parameter no more than $c$.

 We compute the average landscapes of the $H_1$-modules for the two different colourings, Figure \ref{fig:ColouredCircles}. When the large circle has the smaller colour parameter value, the first landscape $(k=1)$ can detect the large circle Figure \ref{subfig:FirstMeanLandscape}. We see the large circle in the first landscape as the large mountain spanning the parameter subspace $[1,5.4]\times[0.5,1.5]$. When the large circle has the higher parameter value, the persistence in the Rips filtration parameter is diminished by the presence of the small circle with smaller colour parameter.
In both colourings, the second landscape $(k=2)$ exhibits the range of parameter values for which both circles are detected Figure \ref{subfig:SecondMeanLandscape}.

We test the robustness of the landscape by repeating the sampling this time with only $50$ points per circle and perturbing both the radii and colour of the sampled points with the addition of i.i.d. normals $\mathcal{N}(0,0.3)$, Figure \ref{subfig:NoisyCirclePointclouds}. We illustrate in Figure \ref{subfig:NoisyFirstMeanLandscape} and Figure \ref{subfig:NoisySecondMeanLandscape} the average landscapes taken over $30$ noisy samples. The resulting landscapes are similar to those of the larger samples without noise.

Let us perform a statistical test to determine whether the multiparameter landscapes can detect that the noisy samples are drawn from different distributions. Consider the functional $f_E(\lambda) = \int_E \lambda  d\mu$. Using the results of Section \ref{Statistics} we find approximate confidence intervals for $f_E(\lambda)$ with $E = \{1\}\times ([2,6]\times[0,1.5]) \subset \mathbb{N}\times \mathbb{R}^2$. We attain approximate $99\%$-confidence intervals on the noisy samples: for Colouring A $[0.400,0.474]$, and for Colouring B $[0.00556,0.00809]$. A two sample $t$-test on the values of this functional on the two sets of colourings attains a $p$-value of $0.00629$. Thus we reject the null hypothesis that the functional values on the landscapes of the two colourings have the same mean.
\begin{figure}[p]
  \centering
  \begin{subfigure}[b]{0.9\linewidth}
    \includegraphics[width=\linewidth]{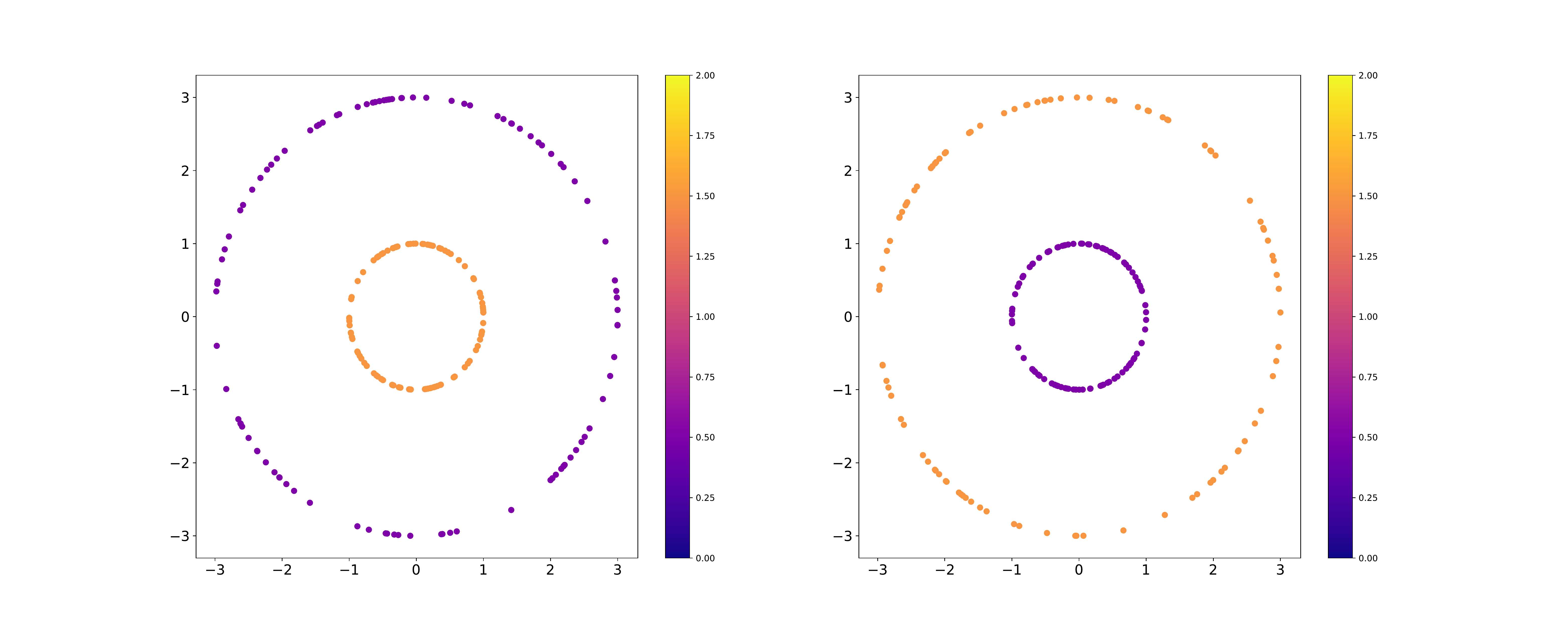}
    \caption{An example point cloud sample from each colouring.}
    \label{subfig:CirclePointclouds}
  \end{subfigure}
  \begin{subfigure}[b]{0.9\linewidth}
    \includegraphics[width=\linewidth]{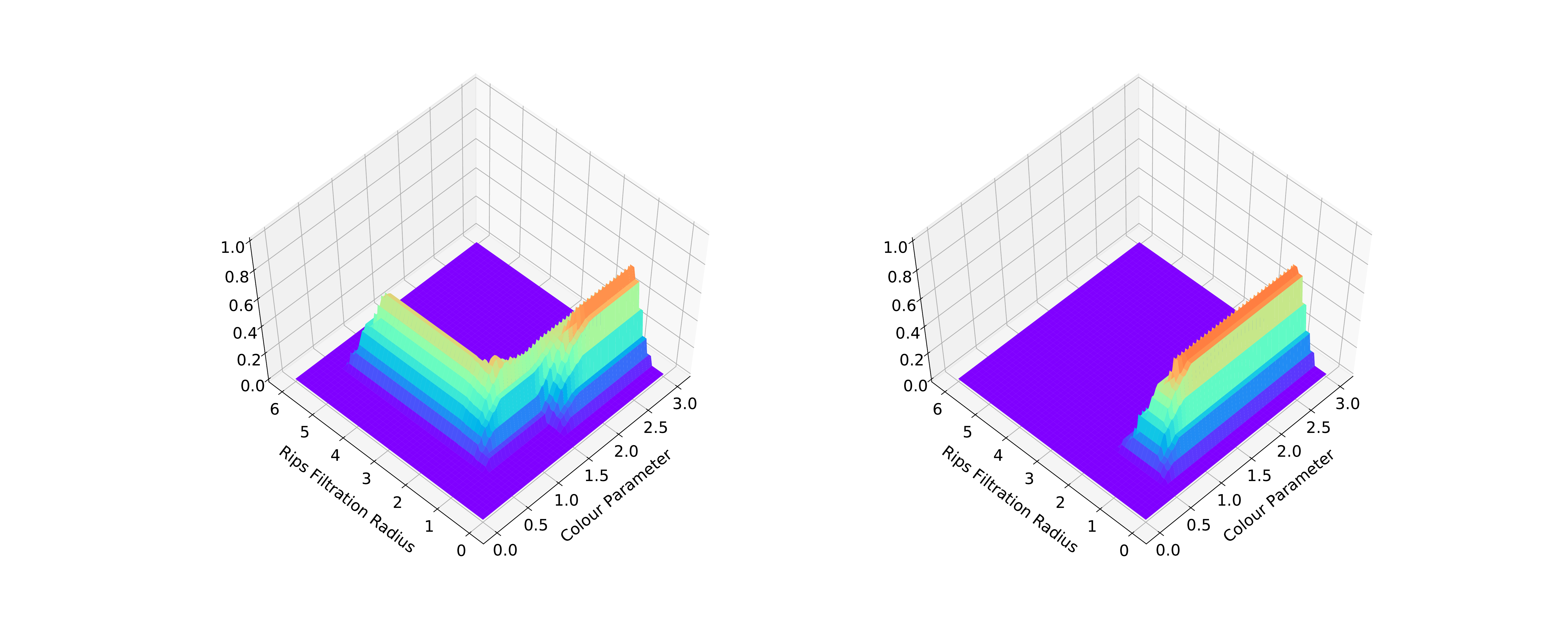}
    \caption{The mean first landscape for each colouring taken over the $30$ samples, $\overline{\lambda_2}(1,\vec{x})$.}
    \label{subfig:FirstMeanLandscape}
  \end{subfigure}
\begin{subfigure}[b]{0.9\linewidth}
    \includegraphics[width=\linewidth]{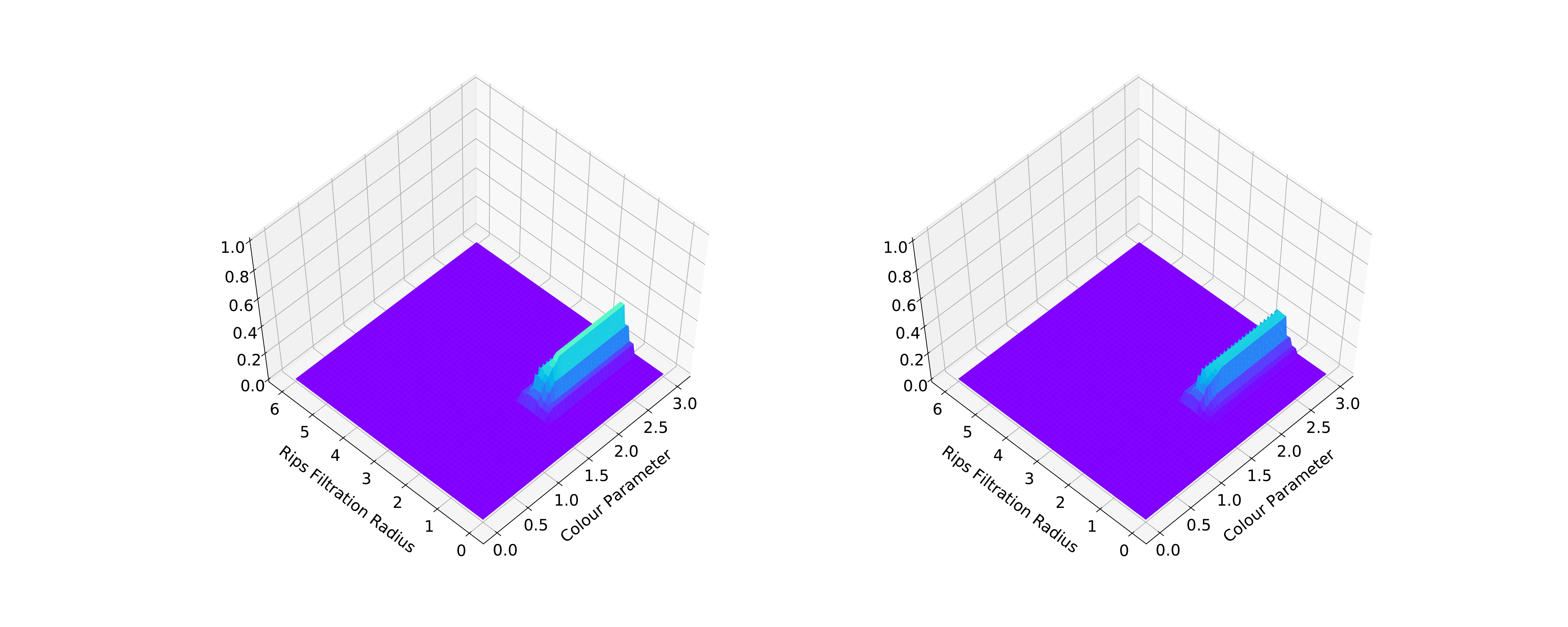}
    \caption{The mean second landscape for each colouring taken over the $30$ samples, $\overline{\lambda_2}(2,\vec{x})$.}
    \label{subfig:SecondMeanLandscape}
  \end{subfigure}
  \caption{The first column shows the plots for Colouring A and the second column Colouring B.}
  \label{fig:ColouredCircles}
\end{figure}

\begin{figure}[p]
  \centering
  \begin{subfigure}[b]{0.9\linewidth}
    \includegraphics[width=\linewidth]{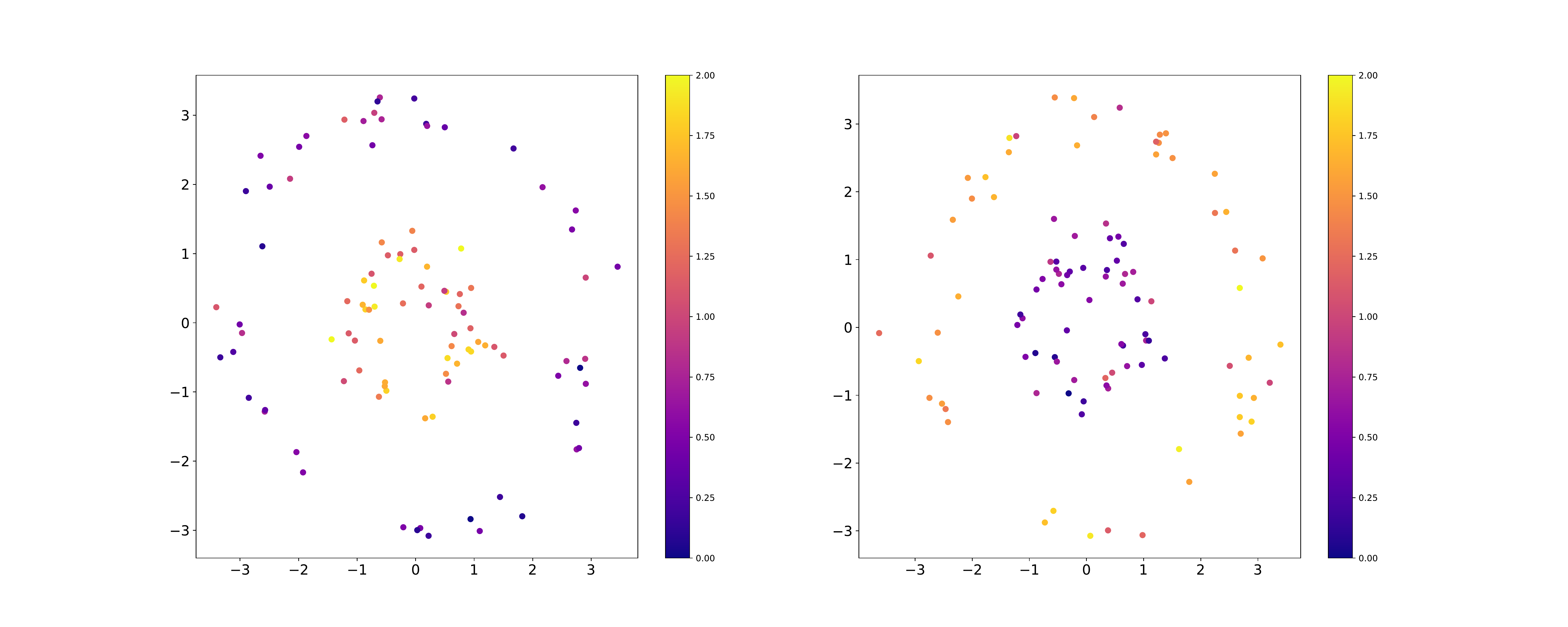}
    \caption{An example point cloud sample from each colouring with noise added}
    \label{subfig:NoisyCirclePointclouds}
  \end{subfigure}
  \begin{subfigure}[b]{0.9\linewidth}
    \includegraphics[width=\linewidth]{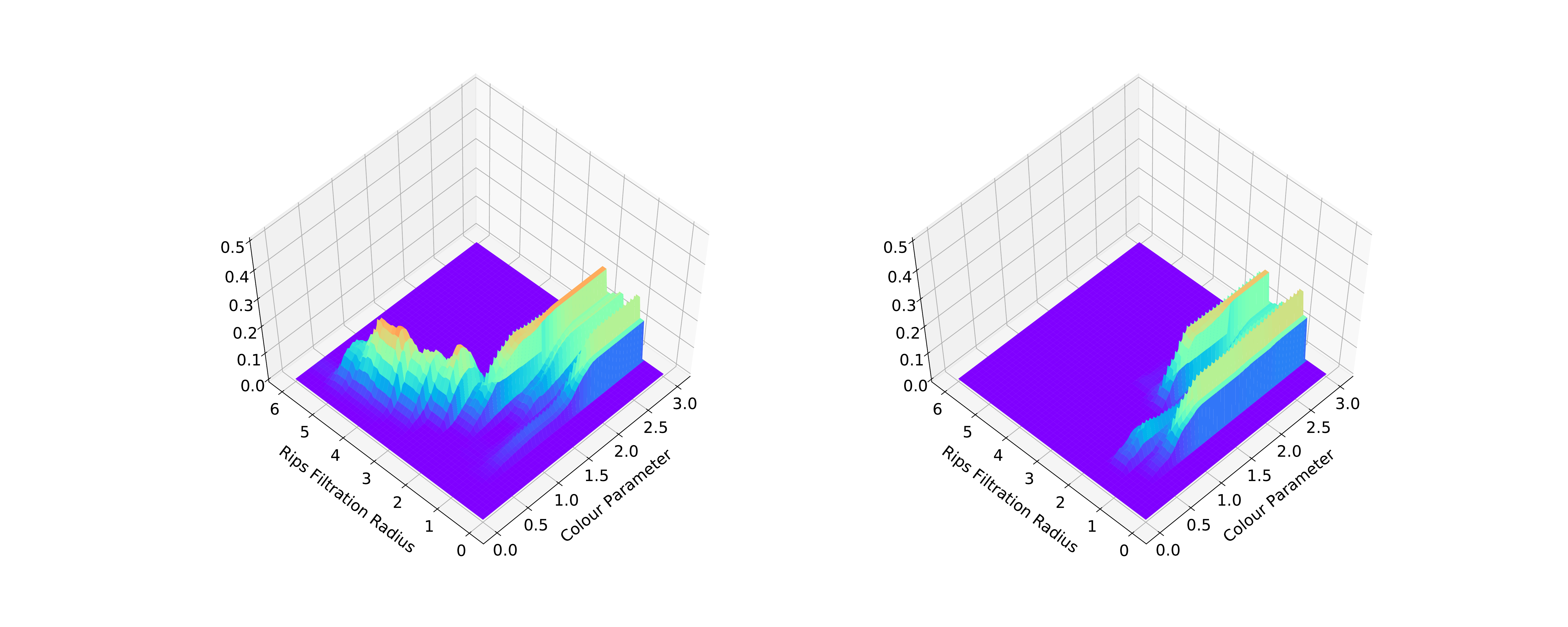}
    \caption{The mean first landscape $\overline{\lambda}_2(1,\vec{x})$ taken over the $30$ noisy samples.}
    \label{subfig:NoisyFirstMeanLandscape}
  \end{subfigure}
   \begin{subfigure}[b]{0.9\linewidth}
    \includegraphics[width=\linewidth]{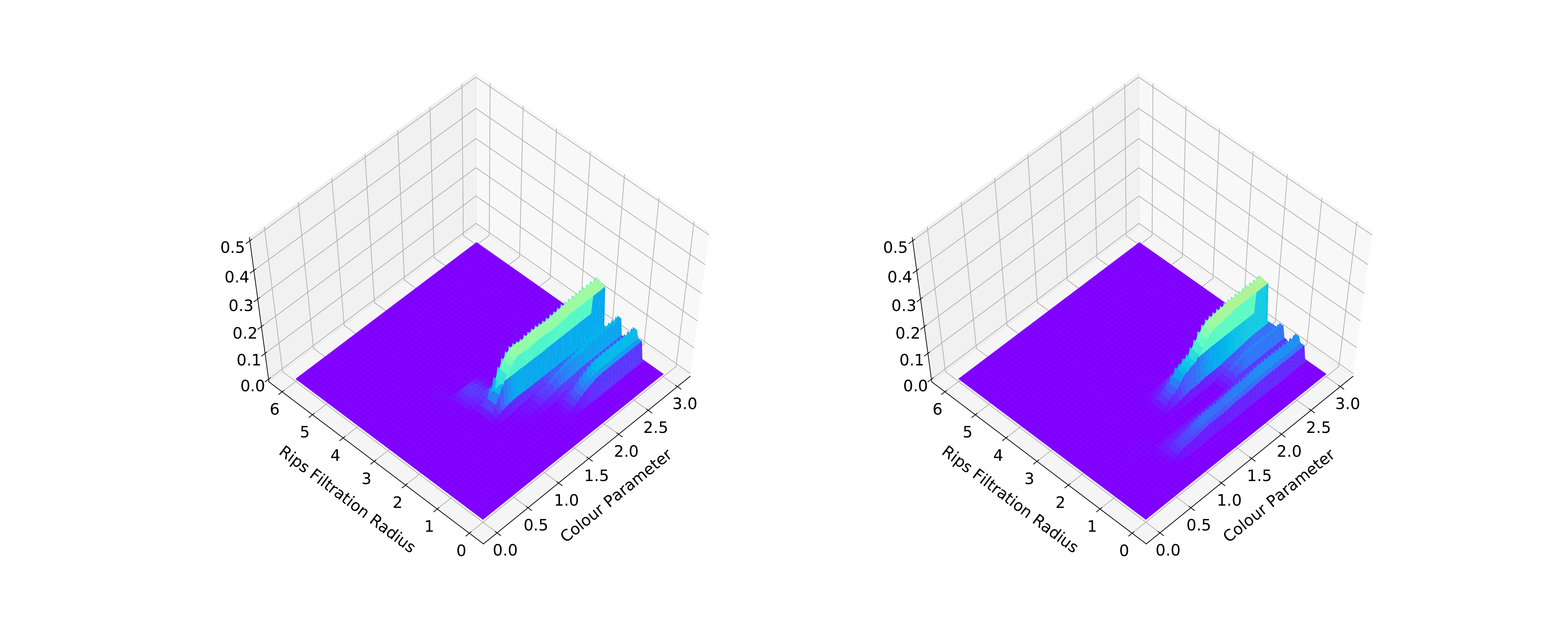}
    \caption{The mean second landscape $\overline{\lambda}_2(1,\vec{x})$ taken over the $30$ noisy samples.}
    \label{subfig:NoisySecondMeanLandscape}
  \end{subfigure}
  \caption{The first column shows the plots for Colouring A and the second column Colouring B.}
  \label{fig:NoisyColouredCircles}
\end{figure}

\subsection{Modal Estimation}

For this example we work on meteorite data which we have lifted from \cite{Good_1980}. The data set consists of values of the proportion of silica measured in $22$ samples. Our task is to infer how many modes there are in the distribution from which this data has been sampled.

A standard approach to this task is kernel density estimation (KDE). With data $\{x_i\} \subset \mathbb{R}^n$ one estimates the probability density function (pdf) of the distribution using a sum of normalised kernels: $$f(x) = \frac{1}{n}\sum_{i=1}^n K_\sigma(x-x_i)$$

Here $K_\sigma$ is a density function with mass concentrated about the origin, for example a Gaussian centred at the origin.
There are two natural parameters in this KDE setup. The \textit{bandwidth} parameter $\sigma$ of the kernel function $K_\sigma$, and a \textit{threshold} parameter which dictates how large a peak in the estimated distribution must be to be considered a mode. The choice of these parameters will dramatically alter our inferred number of modes (see Figure \ref{fig:ChondriteKDEs}).

\begin{figure}[h]
  \centering
    \includegraphics[width=\linewidth]{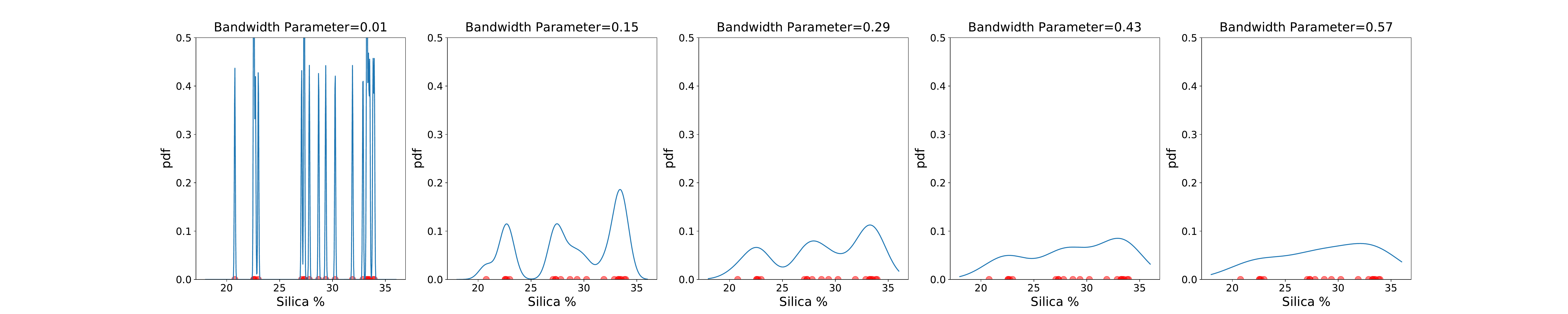}
  \caption{We plot kernel density estimates on the meteorite data (red) for a range of bandwidth parameters. As we increase the bandwidth parameter we yield fewer modes in our kernel density estimate.}
\label{fig:ChondriteKDEs}
\end{figure}

\begin{figure}[h]
  \centering
    \includegraphics[width=0.5\linewidth]{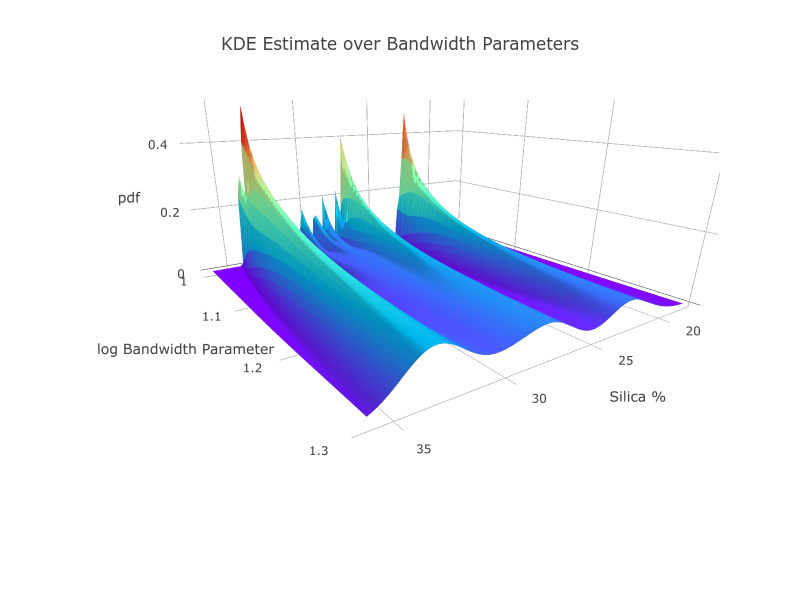}
    \caption{A triangulated surface plot of the KDE for a range of bandwidth parameters. We observe three modes in the KDE estimate for a large range of bandwidth values. }
\label{fig:KDE-Surface}
\end{figure}

Figure \ref{fig:KDE-Surface} is a surface plot of the KDEs ranging over various bandwidth parameters, demonstrating the change in the number of modes as we change the bandwidth. The surface has been triangulated using 
a triangulation subordinate to a regular grid on our parameter space. To each $2$-simplex $\tau$ in the triangulation we attach two parameters; the mean bandwidth $\sigma(\tau)$, and the mean probability density value $p(\tau)$, (averages taken over the vertices of the simplex). We produce a bifiltration by taking the simplicial closure of the $2$-simplices with appropriate parameter values, $X_{(\sigma_0,p_0)} = \text{SC}(\{ \tau | \sigma(\tau) \leq \sigma_0 , p(\tau) \geq 1-p_0 \})$.

The multiparameter landscape detects that three modes appear in the KDEs for a range of parameter values. Looking at the landscapes associated to the $H_0$-module we see that the infinity norm of the first three landscapes is constant but decreases significantly between the third and fourth landscapes, Figure \ref{fig:Modal-Landscapes}. This indicates that within this setup, three modes are seen across a significantly wider range of parameter values than four modes, suggesting the data is drawn from a tri-modal distribution which coincides with our expected result. Whilst in this simple example one could suggest there are three modes from inspection, the landscape analysis can equally be applied to higher dimensional data sets for which visualisation is not possible.

\begin{figure}[h]
  \centering
    \includegraphics[width=\linewidth]{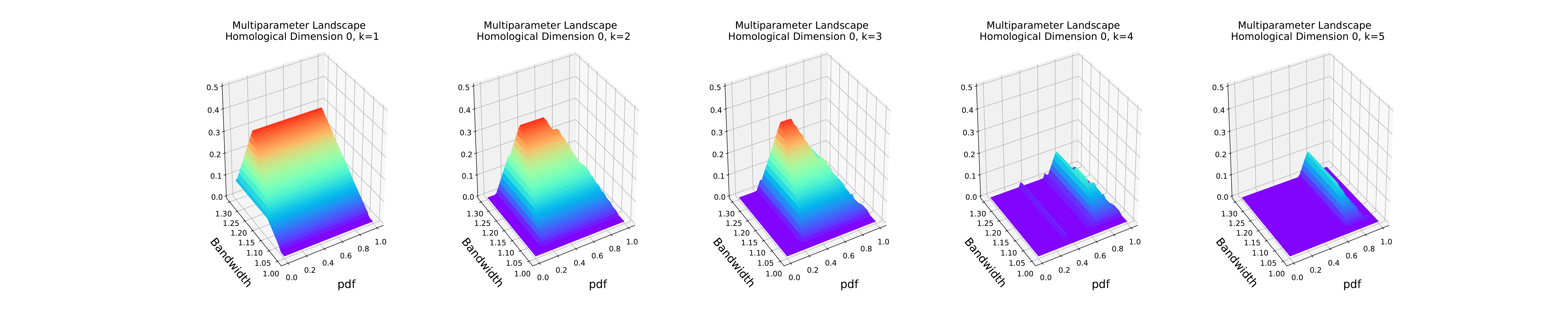}
  \caption{The first to fifth landscapes associated to the $H_0$ module for the KDE surface}
  \label{fig:Modal-Landscapes}
\end{figure}

This basic example can be generalised to detect other properties of KDEs robust to changes in parameter values. For example one could detect significant $i$-dimensional holes in the distribution by considering the $H_i$ module in a similar setup. For related work see Persistence Terraces \cite{Moon2017}.

\subsection{Curvature}

In this subsection we shall work with a synthetic data set sampled from spaces of different curvature. This example is used to emphasise the ability of the multiparameter landscapes to detect geometric differences between point samples. The samples consist of $100$ points chosen uniformly with respect to the volume measure from discs of radius $1$ in the hyperbolic plane, the surface of the unit sphere and Euclidean space so that the spaces have constant curvature of $-1, 1, 0$ respectively. Topologically these disks are all trivial, our landscapes are detecting geometric differences induced by the distribution of points.

We would like to show that the multiparameter landscape is able to detect the curvature of the space from which a sample is drawn given only the pairwise distances between points.


A multifiltered complex is built on the sampled points by filtering the Rips complex with the third nearest neighbour density function $\rho$ on the points. Explicitly, if $\mathcal{P}$ denotes our sampled points and $(r,\rho_0)\in \mathbb{R}^2$ then $X_{(r,\rho_0)} = \text{VR}(\mathcal{P}_{\rho_0},r)$ where $\mathcal{P}_{\rho_0} = \{p\in \mathcal{P} : \rho(p)\leq \rho_0\}$ for the third nearest neighbour density function $\rho$.  We take $100$ samples of $100$ points in each space and investigate the resulting multiparameter landscapes for dimension 1 homology.

We plot the average first multiparameter landscapes in Figure \ref{subfig:CurvatureMeanLandscapes} and the differences between the average landscapes in Figure \ref{subfig:CurvatureLandscapeDifferences}. As one might expect, the persistence of cycles is affected by the curvature of the space. The more negative the curvature the longer the one dimensional cycles persist.

Let us now apply a simple machine learning algorithm to the multiparameter landscapes to see if we can reliably distinguish the curvature of the space from which our small samples have been drawn.

Using the Python package LinearSVC, we train a Support Vector Machine (SVM) with linear kernel on discretizations of the first $10$ landscapes for the samples of the hyperbolic discs and elliptic discs, using $l^2$ penalty and squared hinge loss function. We randomly partition our samples into $160$ training samples and $40$ test samples and evaluate the accuracy by the proportion of test samples correctly classified. Repeating this process $100$ times we attain an average classification score of $85.78\%$. Thus we see that the multiparameter landscapes are able to reliably detect curvature given a relatively small local sample. It is possible that alternative choices of filtration parameters may be better suited to detecting curvature.

\begin{figure}[ht]
  \centering
  \begin{subfigure}[b]{1\linewidth}
    \includegraphics[width=\linewidth]{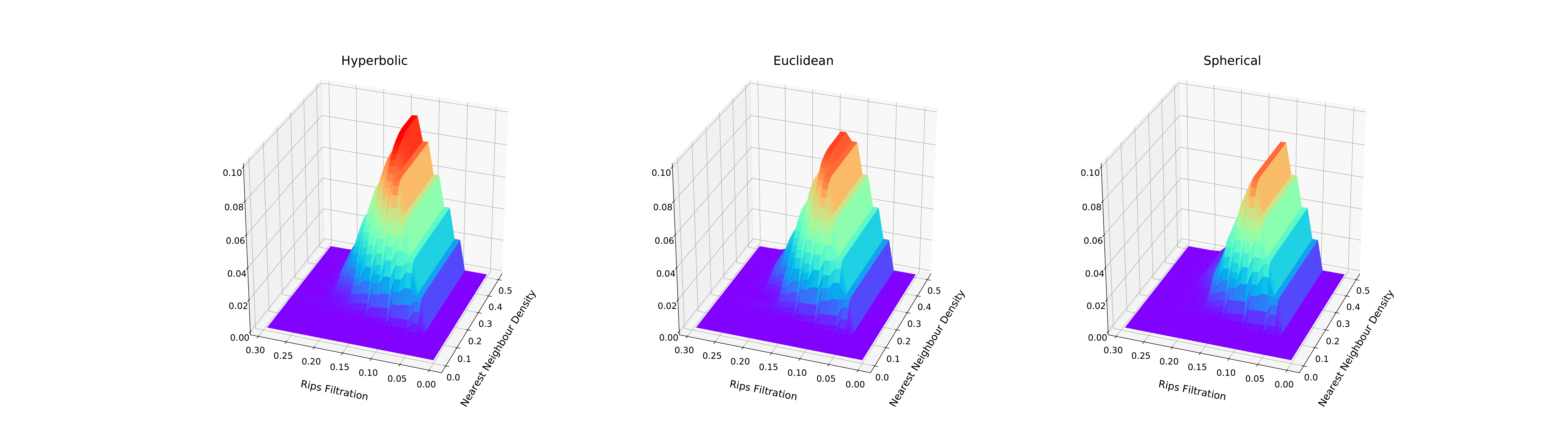}
    \caption{The mean first landscape $\bar{\lambda}_{2}(1,\vec{x})$ of the $H_1$ module for the hyperbolic, Euclidean and elliptic discs taken over $100$ samples.}
    \label{subfig:CurvatureMeanLandscapes}
  \end{subfigure}
  \begin{subfigure}[b]{1\linewidth}
    \includegraphics[width=\linewidth]{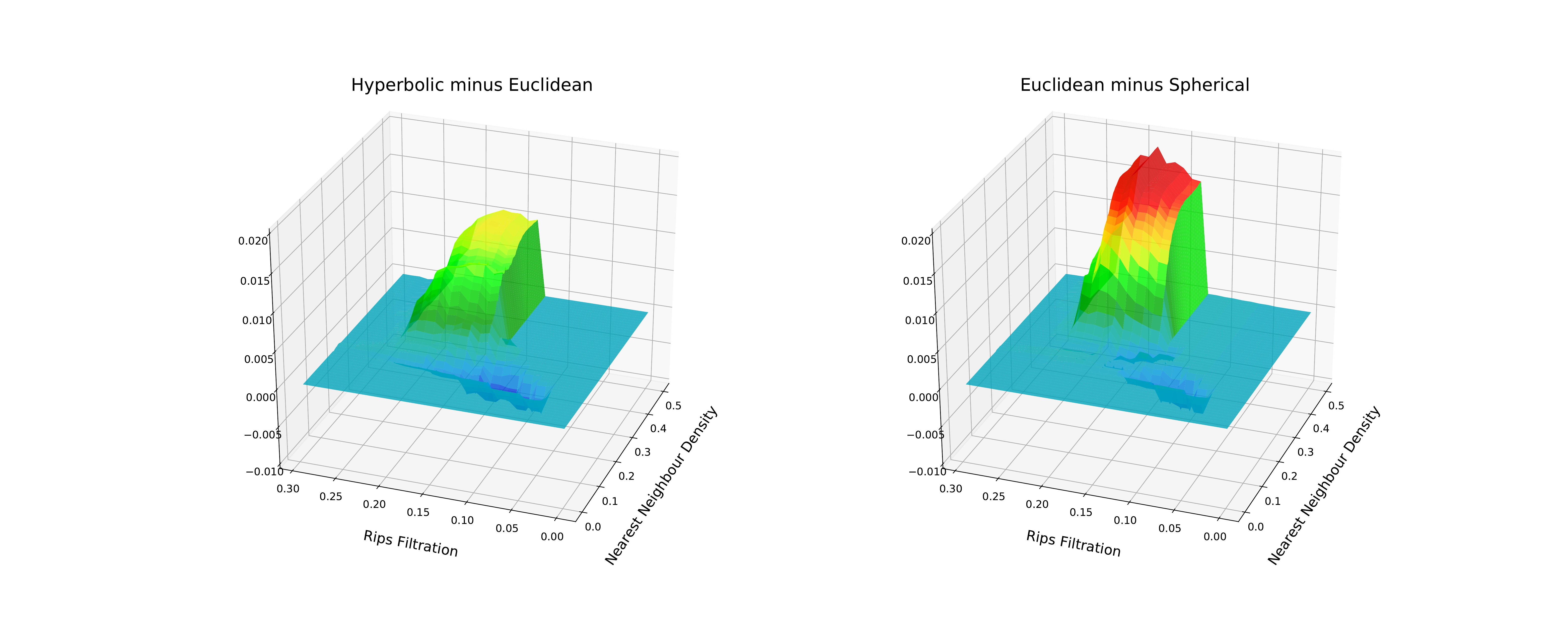}
    \caption{The pointwise difference between the mean landscapes $\bar{\lambda}_{2}(1,\vec{x})$}
    \label{subfig:CurvatureLandscapeDifferences}
  \end{subfigure}
  \caption{}
  \label{fig:Curvature-Average-Landscapes}
\end{figure}

\section{Conclusion}

Multiparameter persistence landscapes provide a stable representation of the rank invariant of a persistence module whilst retaining the discriminating power of the rank invariant. Moreover the landscape distance provides a computable lower bound for the optimal stable distance on persistence modules, the interleaving distance.

The multiparameter landscape also offers a bridge from topological data analysis to machine learning and statistical analysis of multiparameter modules. The multiparameter landscapes, although hard to visualize in dimensions higher than $2$, are interpretable in any dimension with large landscape values indicating features robust to changes in the filtration parameters, and non-zero landscapes for large $k$ indicating a large number of homological features.

The multiparameter landscapes highlight several open questions and challenges in the development of the theory and applications of multiparameter persistent homology that we would be interested to see addressed:
\begin{enumerate}
    \item We would like to understand the relationship between the interleaving distance and landscape distance associated to modules to understand when the landscape distance provides a good lower bound estimate.
    \item We have restricted our invariant to the discriminating power of the rank invariant. We would be interested to see if we could combine our landscapes with invariants that capture the more subtle relationships between features born at incomparable parameter values.
    \item The Bootstrap Method has been used to compute confidence bands for  single parameter persistence landscapes \cite{Chazal2014StochasticCO}. We would be interested in applying similar analysis for multiparameter landscapes.
\end{enumerate}

Finally it is worth remarking that the construction of multiparameter persistence landscapes from multiparameter persistence modules can be generalised to produce stable invariants of generalised persistence modules indexed over other posets. Providing the indexing poset $P$ is equipped with a superlinear family of translations $\Omega$, one can derive a landscape function $\lambda: \mathbb{N}\times P \to \overline{\mathbb{R}}$ from the rank function $\rk : P\times P \to \mathbb{N}$. This landscape equipped with the supremum norm is stable with respect to the interleaving distance induced by the superlinear family, and provides an interpretable, stable representation of the rank function. This vectorization may prove a useful invariant should the computation of generalised persistence modules be developed in future work.

\bibliography{Multiparameter_Persistence_Landscapes}
\bibliographystyle{alpha}

\end{document}